\documentclass[11pt,reqno]{amsart}
\usepackage{amssymb,mathrsfs,color,mathtools,empheq, verbatim, epstopdf}
\usepackage{pinlabel}
\mathtoolsset{showonlyrefs}
\usepackage{hyperref} 

\usepackage{enumerate, bbm}
\usepackage{tensor}

\usepackage{graphicx}
\usepackage{xcolor} 
\usepackage{tensor}
\usepackage{slashed}
\usepackage{cite}

\usepackage[shortlabels]{enumitem}

\usepackage{geometry}

\numberwithin{equation}{section}

\newtheorem{mainthm}{Theorem}
\newtheorem{thm}{Theorem}[section]
\newtheorem{cor}[thm]{Corollary}
\newtheorem{lem}[thm]{Lemma}
\newtheorem{prop}[thm]{Proposition}

\theoremstyle{definition} 
\newtheorem{rem}[thm]{Remark}
\newtheorem{defn}[thm]{Definition}

\theoremstyle{remark}

\def\bN {\mathbb{N}}

\def\bR {\mathbb{R}}
\def\bS {\mathbb{S}}

\def\bZ {\mathbb{Z}}

\def\cE {\mathcal{E}}

\newcommand{\La}{\big\langle}
\newcommand{\Ra}{\big\rangle}

\newcommand{\bs}[1]{\boldsymbol{#1}}
\newcommand{\vd}{\mathrm{d}}

\newcommand{\lin}{_{\textsc{l}}}

\definecolor{deepgreen}{cmyk}{1,0,1,0.5}

\newcommand{\E}{\mathcal{E}}

\newcommand{\LL}{\mathcal{L}}

\newcommand{\N}{\mathbb{N}}
\newcommand{\R}{\mathbb{R}}
\newcommand{\Sp}{\mathbb{S}}
\newcommand{\Z}{\mathbb{Z}}

\newcommand{\al}{\alpha}

\newcommand{\de}{\delta}

\newcommand{\lam}{\lambda}
\newcommand{\te}{\theta}

\newcommand{\s}{\sigma}

\newcommand{\De}{\Delta}

\newcommand{\Lam}{\Lambda}

\newcommand{\p}{\partial}
\newcommand{\na}{\nabla}

\makeatletter

\newcommand{\Rmnum}[1]{\expandafter\@slowromancap\romannumeral #1@}
\makeatother

\newcommand{\ti}{\widetilde}

\newcommand{\U}{\underline}

\newcommand{\ang}[1]{\left\langle{#1}\right\rangle}
\newcommand{\abs}[1]{\left\lvert{#1}\right\rvert}

\newcommand{\EQ}[1]{\begin{equation}\begin{split} #1 \end{split}\end{equation}}

\newcommand{\Del}[1]{}

\newcommand{\mand}{{\ \ \text{and} \ \  }}

\newcommand{\mif}{{\ \ \text{if} \ \ }}
\newcommand{\mfor}{{\ \ \text{for} \ \ }}
\newcommand{\mas}{{\ \ \text{as} \ \ }}

\newcommand{\uD}{\operatorname{D}}

\definecolor{green}{rgb}{0,0.8,0} 

\newcommand{\ud}{\mathrm{d}}

\newcommand{\eps}{\epsilon}

\newcommand{\bfd}{{\bf d}}
\newcommand{\bfe}{{\bf e}}

\newcommand{\bfi}{{\bf i}}

\newcommand{\bfp}{{\bf p}}
\newcommand{\bfq}{{\bf q}}

\newcommand{\calA}{\mathcal A}

\newcommand{\calC}{\mathcal C}

\newcommand{\calE}{\mathcal E}

\newcommand{\calI}{\mathcal I}
\newcommand{\calJ}{\mathcal J}

\newcommand{\calL}{\mathcal L}
\newcommand{\calM}{\mathcal M}

\newcommand{\calQ}{\mathcal Q}

\newcommand{\calT}{\mathcal T}

\newcommand{\calZ}{\mathcal Z}

\newcommand{\ULam}{\U{\Lam}}

\begin{document}

\title[Bubbling for the harmonic map heat flow]{Bubble decomposition for the harmonic map heat flow \\in the equivariant case}
\author{Jacek Jendrej}
\author{Andrew Lawrie}

\begin{abstract}
We consider the harmonic map heat flow for maps $\bR^{2} \to \bS^2$, under equivariant symmetry.  It is known that solutions to the initial value problem can exhibit bubbling along a sequence of times -- the solution decouples into a superposition of harmonic maps concentrating at different scales and a body map that accounts for the rest of the energy.  We prove that this bubble decomposition is unique and occurs continuously in time. The main new ingredient in the proof is the notion of a collision interval from~\cite{JL6}.  
\end{abstract}


\thanks{J.Jendrej is supported by  ANR-18-CE40-0028 project ESSED.  A. Lawrie is supported by NSF grant DMS-1954455, a Sloan Research Fellowship, and the Solomon Buchsbaum Research Fund. 
}

\maketitle

\tableofcontents

\section{Introduction}

\subsection{Setting of the problem}

Consider the harmonic map heat flow (HMHF) for maps $\Psi: \R^2 \to \Sp^2 \subset \R^3$, that is,  the heat flow associated to the Dirichlet energy 
\EQ{
E(\Psi) := \frac{1}{2} \int_{\R^2} \abs{\na \Psi(x)}^2 \, \ud x.
}
The initial value problem for the HMHF is given by
\EQ{ \label{eq:hmhf} 
\p_t \Psi - \De \Psi  &= \Psi \abs{\na \Psi}^2  \\
\Psi(0, x) &= \Psi_0(x). 
}
 We say a solution to~\eqref{eq:hmhf} is  $k$-equivariant if it takes the form 
\EQ{
\Psi(t, re^{i\te}) = (\sin u(t, r) \cos k \te,  \sin  u(t, r) \sin k \te,  \cos u(t, r)) \in \Sp^2 \subset \R^3,
}
where $k \in \N$ and $(r, \te)$ are polar coordinates on $\R^2$. In this case the HMHF reduces to a scalar equation for the polar angle $u= u(t,r)$, 
\EQ{ \label{eq:hf} 
\p_t u &= \p_{r}^2 u + \frac{1}{r} \p_r u - \frac{k^2}{r^2} \frac{\sin 2u}{2 }, \\
u(0) &= u_0,
}
and the energy $E  = E(u)$ reduces to \EQ{
&E( u(t) )  =   2 \pi   \int_0^\infty \frac{1}{2}\left((\p_r u(t, r))^2 + k^2 \frac{\sin^2 (u(t, r))}{r^2} \right)\, r  \, \ud r ,
}
and formally satisfies 
\EQ{
& \frac{\ud}{\ud t} E( u(t))  = - 2 \pi \int_0^\infty (\p_t u(t, r))^2 \, r \, \ud r  = - 2 \pi  \| \calT(u(t)) \|_{L^2}^2 ,
}
where in the $k$-equivariant setting $\calT(u) :=  \p_{r}^2 u + \frac{1}{r} \p_r u - \frac{k^2}{2r^2} \sin(2u)$ is called the tension of $u$. Integrating in time from $t_0$ to $t$ gives, 
\EQ{ \label{eq:energy-identity-intro} 
E( u(t)) + 2 \pi \int_{t_0}^t  \| \calT(u(s)) \|_{L^2}^2 \, \ud s =  E( u(t_0)) .
}

The natural setting in which to consider the initial value problem for~\eqref{eq:hf} is the space of initial data $u_0$ with finite energy, $E(u) < \infty$. This set is split into disjoint sectors, $\E_{\ell, m}$, which for $\ell, m\in \bZ$, are defined by 
\EQ{
\calE_{\ell, m}:= \big\{ u \mid E(u) < \infty, \quad \lim_{r \to 0} u(r) = \ell\pi, \quad \lim_{r \to \infty} u(r) = m \pi \big\}.
}
These sectors, which are preserved by the flow, are related to the topological degree of the full map $\Psi: \bR^2 \to  \bS^2$:
 if $m - \ell$ is even and $u \in \E_{\ell, m}$, then the corresponding map $\Psi$ with polar angle $u$ is topologically trivial,
 whereas for odd $m - \ell$ the map has degree $k$.

The sets $\calE_{\ell, m}$ are affine spaces, parallel to the linear space $\E := \E_{0, 0}$, which we endow with the norm, 
\EQ{
\|  u_0 \|_{\E}^2:= \int_0^\infty \Big( ( \p_r u_0(r))^2 + k^2 \frac{ (u_0(r))^2}{r^2} \Big) \, r \ud r.
} 
We make note of the embedding $\|u_0 \|_{L^\infty} \le C \| u_0 \|_{\E}$. 

The unique $k$-equivariant harmonic map is given  explicitly by
\EQ{
Q(r) := 2 \arctan (r^k).
}
Here uniqueness means up to scaling,  sign change, and adding a multiple of $\pi$, i.e.,   every finite energy stationary solution to~\eqref{eq:hf}  takes the form $Q_{\mu, \sigma, m}(r) =  m \pi + \sigma Q(r/ \mu)$ for some  $\mu \in (0, \infty), \sigma \in \{0, -1, 1\}$ and $m \in \Z$. 
The map $ Q$ and its rescaled versions $ Q_\lambda(r) := Q(\lambda^{-1}r)$ for $\lambda > 0$,
are minimizers of the energy $E$  within the class $\E_{0, 1}$;  in fact, $E(  Q_\lambda ) = 4  \pi k$.

\subsection{Statement of the results}

We prove the following theorem. 

\begin{mainthm}[Bubble decomposition]\label{thm:main} 
Let $k \in \N$, let $\ell, m \in \Z$,  and let $u(t)$ be the solution to~\eqref{eq:hf} with initial data $u(0) = u_0 \in \E_{\ell, m}$,  defined on its maximal  interval of existence $[0,T_+)$. 
 
 \emph{({Global solution})} If $T_+ = \infty$, there exist a time $T_0>0$,  an integer $N \ge 0$, continuous functions $\lam_1(t), \dots,  \lam_N(t) \in C^0([T_0, \infty))$, signs $\iota_1, \dots, \iota_N \in \{-1, 1\}$,  and $ g(t) \in \E$ defined by 
 \EQ{ \label{eq:sr-global} 
  u(t) = m  \pi + \sum_{j =1}^N \iota_j ( Q_{\lam_j(t)} -  \pi) +  g(t) , 
 }
 such that 
 \EQ{
 \|  g(t)\|_{\E} + \sum_{j =1}^{N} \frac{\lam_{j}(t)}{\lam_{j+1}(t)}  \to 0 \mas t \to \infty,  
 }
 where above we use the convention that $\lam_{N+1}(t) = \sqrt{t}$.

 \emph{({Blow-up solution})} If $T_+ < \infty$, there exist a time $T_0< T_+$,   integers $m_{\infty}, m_\De$, a mapping $u^*\in \E_{0, m_{\infty}}$, an integer $N \ge 1$, continuous functions $\lam_1(t), \dots,  \lam_N(t) \in C^0([T_0, T_+))$, signs $\iota_1, \dots, \iota_N \in \{-1, 1\}$, and $g(t) \in \E$ defined by 
 \EQ{ \label{eq:sr-bu} 
 u(t) = m_\De  \pi + \sum_{j =1}^N \iota_j( Q_{\lam_j(t)} - \pi) +  u^* +  g(t) , 
 }
 such that 
 \EQ{
 \|  g(t)\|_{\E} + \sum_{j =1}^{N} \frac{\lam_{j}(t)}{\lam_{j+1}(t)}  \to 0 \mas t \to T_+, 
 }
 where above we use the convention that $\lam_{N+1}(t) = \sqrt{T_+-t}$. 
\end{mainthm} 

\begin{rem} 
Asymptotic decompositions of solutions to~\eqref{eq:hf} (in fact for solutions to the equation~\eqref{eq:hmhf} without symmetry assumptions) were proved along a \emph{sequence of times} $t_n \to T_+$,  in a series of works by Struwe~\cite{Struwe85},  Qing~\cite{Qing}, Ding-Tian~\cite{DT}, Wang~\cite{Wang}, Qing-Tian~\cite{QT}, and Topping~\cite{Topping-winding}.  The main contribution of this paper is to show that the decomposition can be taken continuously in time for $k$-equivariant solutions. 



\end{rem} 

\begin{rem}
 \label{Topping} 
In the non-equivariant setting, i.e., for~\eqref{eq:hmhf},  Topping~\cite{Top-JDG, Topping04} made important progress on a related question in the global case, showing the uniqueness of the locations of the bubbling points under restrictions on the configurations of bubbles appearing in the sequential decomposition. His assumption, roughly, is that all of the bubbles concentrating at a certain point have to have the same orientation. We can contrast this assumption with the equivariant setting, where in the decomposition~\eqref{eq:sr-global} subsequent bubbles have opposite orientations. 

\end{rem} 

\begin{rem} 
Given Theorem~\ref{thm:main}, it is natural to ask which configurations of bubbles are possible in the decomposition.  Van der Hout~\cite{vdHout03} showed that there can only be one bubble in the decomposition in the case of equivariant  finite time blow-up; see also~\cite{BvdHH}. In contrast, in the infinite time case, it is expected that there can be equivariant bubble trees of arbitrary size (see recent work of Del Pino, Musso, and Wei~\cite{DPMW} for a construction in the case of the critical semi-linear heat equation). 
\end{rem} 

\begin{rem} 
There are solutions to the HMHF that develop a bubbling singularity in finite time, the first being the examples of Coron and Ghidaglia~\cite{CG} (in dimension $d \ge 3$) and Chang, Ding, Ye~\cite{CDY} in the $2d$ case considered here.  Guan, Gustafson, and Tsai~\cite{GGT} and Gustafson, Nakanishi, and  Tsai~\cite{GNT} showed that the harmonic maps $Q$ are asympotically stable in equivariance classes $k \ge 3$, and thus there is no finite time blow up for energies close to $Q$ in that setting. For $k=2$,~\cite{GNT} gave examples of solutions exhibiting infnite time blow up and eternal oscillations. Rapha\"el and Schweyer constructed a stable blow-up regime for $k=1$ in~\cite{RSc-13} and then blow up solutions with different rates in~\cite{RSc-14}. 
Recently,  Davila, Del Pino, and Wei~\cite{DDPW} constructed examples of solutions simultaneously concentrating a single copy of the ground state harmonic map at distinct points in space. 
\end{rem}

\subsection{Summary of the proof}

We give an informal description of the proof of Theorem~\ref{thm:main} starting with a summary of the sequential bubbling results as in, e.g.,~\cite{Qing, Topping-winding}, adapted to our setting. 
A crucial ingredient is a sequential compactness lemma, which says that a sequence of maps with vanishing tension must converge (at least locally in space) to a multi-bubble, which we define as follows.  

\begin{defn}[Multi-bubble configuration] \label{def:multi-bubble} 
Given $M \in \{0, 1, \ldots\}$, $m \in \bZ$, $\vec\iota = (\iota_1, \ldots, \iota_M) \in \{-1, 1\}^M$
and an increasing sequence $\vec\lambda = (\lambda_1, \ldots, \lambda_M) \in (0, \infty)^M$,
a \emph{multi-bubble configuration} is defined by the formula
\begin{equation}
\calQ(m, \vec\iota, \vec\lambda; r) := m\pi + \sum_{j=1}^M\iota_j\big( Q_{\lambda_j}(r) - \pi\big).
\end{equation}
\end{defn}
\begin{rem}
If $M = 0$, it should be understood that $\calQ(m, \vec\iota, \vec\lambda; r) = m\pi$
for all $r \in (0, \infty)$, where $\vec\iota$ and $\vec\lambda$ are $0$-element sequences,
that is the unique functions $\emptyset \to \{-1, 1\}$ and $\emptyset \to (0, \infty)$, respectively. 
\end{rem}

With this definition, we  define a localized distance  function to multi-bubble configurations by 
\EQ{ \label{eq:delta-def-intro} 
\bs\delta_{R}(u) := \inf_{m, M, \vec \iota, \vec \lam} \Big( \| u - \calQ(m, \vec \iota, \vec \lam) \|_{\E(r \le R)}^2 + \sum_{j =1}^M \Big( \frac{\lam_j}{\lam_{j+1}} \Big)^k \Big)^{\frac{1}{2}}
}
where the infimum is taken over all $m \in \Z$, $M \in \{0, 1, 2, \dots\}$, all vectors $\iota \in\{-1, 1\}^M$, $\vec \lam \in (0, \infty)^M$, and  we use the convention that the last scale $\lam_{M+1} = R$. 

The localized sequential compactness lemma (see Lemma~\ref{lem:compact}) says the following: given a sequence of maps $u_n$  with bounded energy,  a sequence $\rho_n \in (0, \infty)$ of scales,  and tension vanishing in $L^2$ relative to the scale $\rho_n$, i.e., $ \lim_{n \to \infty}  \rho_n \| \calT(u_n) \|_{L^2} = 0$, there exists  a subsequence of the $u_n$ that converges to  a multi-bubble configuration up to large scales relative to $\rho_n$, i.e., $ \lim_{n \to \infty} \bs \de_{R_n \rho_n} ( u_n)  = 0$ for some sequence $R_n \to \infty$. An analogous result with no symmetry assumptions was proved by  Qing~\cite{Qing} using the local bubbling theory of Struwe~\cite{Struwe85} together with a delicate elliptic analysis showing that no energy can accumulate on the ``neck'' regions between the bubbles. Here we give a mostly self-contained proof of this compactness result in the simpler equivariant setting using the theory of profile decompositions of G\'erard~\cite{Gerard} and an approach in the spirit of Duyckaerts, Kenig, and Merle's work on nonlinear waves~\cite{DKM3}. To control the energy on the neck regions we use a virial-type functional adapted from Jia and Kenig's proof of sequential soliton resolution for equivariant wave maps~\cite{JK}.

With the compactness lemma in place, we now consider the heat flow. 
To fix ideas, let $u(t)$ be a solution to~\eqref{eq:hf} defined globally in time, i.e., $T_+ = \infty$. By the energy identity~\eqref{eq:energy-identity-intro}, 
\EQ{ \label{eq:bt} 
\int_0^\infty \| \calT(u(t)  \|_{L^2}^2 \, \ud t < \infty, 
}
and thus we can find a sequence of times $t_n \to \infty$ so that $ \lim_{n \to \infty}\sqrt{ t_n} \| \calT(u(t_n)) \|_{L^2}  = 0$. From the compactness lemma we deduce that after passing to a subsequence of the $t_n$, $u(t_n)$ converges to an $N$-bubble configuration up to the self-similar scale $r = \sqrt{t_n}$. In the exterior region $r \gtrsim \sqrt{t}$, we prove that $u(t)$ has vanishing energy (continuously in time) using a localized energy inequality due to Struwe~\cite{Struwe85}; see Proposition~\ref{prop:seq-global}. 

Let $\bfd(t)$ denote the distance to the  particular $N$-bubble configuration obtained via the compactness lemma  (which is defined analogously to~\eqref{eq:delta-def-intro},  except without the spatial localization; see Definition~\ref{def:proximity}). We have so far proved that 
\EQ{
\lim_{n \to \infty} \bfd(t_n)  = 0. 
}
Theorem~\ref{thm:main} follows from showing that in fact $\lim_{t \to \infty} \bfd(t)  = 0$. We assume that continuous-in-time convergence of $\bfd(t)$ fails. To reach a contradiction we study time intervals on which bubbles come into collision (i.e., where $\bfd(t)$ grows), adapting the notion of a \emph{collision interval} from our paper~\cite{JL6}. 

We say that an interval $[a, b]$ is a collision interval with parameters $0<\eps< \eta$ and $N-K$ exterior bubbles for some $1 \le K \le N$, if $\bfd(a) \le \eps$,  $\bfd(b) \ge \eta$, and there exists a curve $r = \rho_K(t)$ outside of which $ u(t)$ is within $\eps$ of an $N-K$-bubble (in the sense of a localized version of $\bfd(t)$); see Defintion~\ref{def:collision}.  We now define $K$ to be the \emph{smallest} non-negative integer for which there exists $\eta>0$, a sequence $\eps_n \to 0$,  and sequences $a_n, b_n \to \infty$, so that $[a_n, b_n]$ are collision intervals with parameters $\eps_n, \eta$ and $N-K$ exterior bubbles, and we write $[a_n, b_n] \in \calC_K( \eps_n, \eta)$; see Section~\ref{ssec:collision} for the proof that $K$ is well-defined and $\ge 1$, under the contradiction hypothesis.  


Consider a sequence of collision intervals $[a_n, b_n] \in \calC_K( \eps_n, \eta)$. Near the endpoint $a_n$,  $u(t)$ is close to an $N$-bubble configuration and we denote the interior scales, which will come into collision, by $\vec \lam = ( \lam_1, \dots, \lam_K)$ and the exterior scales, which stay coherent, by $\vec \mu = ( \vec \mu_{K+1}, \dots, \vec \mu_N)$. The crucial point is that the minimality of $K$ allows us to relate the scale of the $K$th bubble $\lam_K$ to the lengths of the collision intervals $b_n - a_n$. We prove, roughly,  that for sufficiently large $n$ the collision intervals $[a_n, b_n]$  contain subintervals $[c_n, d_n]$ on which (1) $\inf_{t \in [c_n, d_n]}\bfd(t)  \ge \alpha$ for some $\alpha>0$, (2) the scale $\lam_K(t)$ stays roughly constant on $[c_n, d_n]$, and (3) the lower bound $d_n - c_n \gtrsim n^{-1} \lam_K(c_n)^2$ holds. The compactness lemma and the lower bound $\bfd(t) \ge \alpha$ together yield a lower bound on the tension $\inf_{t \in [c_n, d_n]}\lam_K(c_n)^2\| \calT( u(t)  \|_{L^2}^2 \gtrsim 1$ where the scale $\lam_K$ appears again due to the definition of $K$. The last two sentences lead to an immediate contradiction from the boundedness of the integral~\eqref{eq:bt}, i.e.,  
\EQ{
C \ge \int_0^\infty \| \calT(u(t)  \|_{L^2}^2 \, \ud t  \ge \sum_n \int_{c_n}^{d_n} \| \calT(u(t)  \|_{L^2}^2 \, \ud t  \gtrsim \sum_n  n^{-1}, 
}
which proves that $\lim_{t \to \infty} \bfd(t)= 0$.

\subsection{Notational conventions}
The energy is denoted $E$, $\cE$ is the energy space, $\cE_{\ell, m}$ are the finite energy sectors. We use the notation $\cE(r_1, r_2)$ to denote the local energy norm
\begin{equation}
\| g\|_{\cE(r_1, r_2)}^2 := \int_{r_1}^{r_2} \Big( (\partial_r g)^2 + \frac{k^2}{r^2}g^2\Big)\,r\vd r,
\end{equation}
By convention, $\cE(r_0) := \cE(r_0, \infty)$ for $r_0 > 0$. The local nonlinear energy is denoted $E(\bs u_0; r_1, r_2)$. 
We adopt similar conventions as for $\cE$ regarding the omission
of $r_2$, or both $r_1$ and $r_2$.

Given a function $\phi(r)$ and $\lambda>0$, we denote by $\phi_{\lam}(r) = \phi(r/ \lam)$, the $\E$-invariant re-scaling, and by $\phi_{\U{\lam}}(r) = \lam^{-1} \phi(r/ \lambda)$ the $L^2$-invariant re-scaling. We denote by $\Lam :=r \p_r$ and $\ULam := r \partial r +1$ the infinitesimal generators of these scalings. We denote $\ang{\cdot\mid\cdot}$
the radial $L^2(\bR^2)$ inner product given by,  
\EQ{
\La \phi \mid \psi \Ra := \int_0^\infty \phi(r) \psi(r) \,r \, \ud r. 
}
We denote $k$ the equivariance degree and $f(u) := \frac{1}{2} \sin 2u$ the nonlinearity in \eqref{eq:hf}.
We let $\chi$ be a smooth cut-off function, supported in $r \leq 2$ and equal $1$ for $r \le 1$.

We call a ``constant'' a number which depends only on the equivariance degree $k$ and the number of bubbles $N$.
Constants are denoted $C, C_0, C_1, c, c_0, c_1$. We write $A \lesssim B$ if $A \leq CB$ and $A \gtrsim B$ if $A \geq cB$.
We write $A \ll B$ if $\lim_{n\to \infty} A / B = 0$.

For any sets $X, Y, Z$ we identify $Z^{X\times Y}$ with $(Z^Y)^X$, which means that
if $\phi: X\times Y \to Z$ is a function, then for any $x \in X$ we can view $\phi(x)$ as a function $Y \to Z$
given by $(\phi(x))(y) := \phi(x, y)$.

\section{Preliminaries}

\subsection{Well-posedness} 

The starting point for our analysis is the following result of Struwe~\cite{Struwe85}, which says that the initial value problem for the harmonic map flow is well-posed for data in the energy space.

\begin{lem}[Local well-posedness]\cite[Theorem 4.1]{Struwe85} For each $\ell, m \in \Z$ and  $u_0\in \E_{\ell, m}$ there exists a maximal time of existence $T_+= T_+(u_0)$ and a unique solution $u(t) \in \E_{\ell, m}$  to~\eqref{eq:hf} on the time interval $t \in [0, T_+)$ with $u(0)= u_0$. The maximal time is characterized by the following condition: if $T_+<\infty$, there exists $\eps_0>0$ such that 
\EQ{\label{eq:bu-crit} 
\limsup_{ t \to T_+} E(u(t); 0, r_0) \ge \eps_0,
}
for all $r_0>0$. If there is no such $T_+<\infty$, we say $T_+ = \infty$ and the flow is globally defined.

The energy $E(u(t))$ is absolutely continuous and non-increasing as a function of $t \in [0, T]$ for any $T < T_+$,  and for any $t_1 \le t_2 \in [0, T_+)$,  there holds, 
\EQ{ \label{eq:energy-identity} 
E(u(t_2)) + 2\pi \int_{t_1}^{t_2} \int_0^\infty ( \p_t u(t, r))^2 \,r \, \ud r \ud t = E(u(t_1)). 
}
In particular, 
\EQ{ \label{eq:tension-L2} 
\int_0^{T_+} \int_0^\infty ( \p_t u(t, r))^2 \,r \, \ud r \ud t \le E(u_0). 
}
\end{lem} 

\begin{rem} 
Local well-posedness is proved by Struwe for the HMHF without symmetry assumptions in the case of maps from a closed Riemann surface $\calM \to \Sp^2$. For the case of maps from $\R^2$ we refer the reader to Lin and Wang~\cite[Theorem 5.2.1]{Lin-Wang} for the short time existence of  regular solutions. As equivariant symmetry is preserved by the flow, we obtain regular equivariant solutions to~\eqref{eq:hf} by taking equivariant initial data. Solutions with finite energy initial data are then obtained as limits of smooth solutions, and in~\cite{Struwe85} Struwe proved these solutions are regular, e.g., $C^2$, on any compact time interval $[\tau, T] \subset (0, T_+)$. We note that in the equivariant case the energy can only concentrate at the origin $r=0$, giving the form of the blow-up criterion in~\eqref{eq:bu-crit}. 
\end{rem}

\subsection{Basic estimates}

\begin{lem}  \label{lem:pi} 
Fix integers $\ell, m$. For every $\eps>0$ and $R_0 >1$,  there exists a $\de>0$ with the following property. Let $0 \le R_1 < R_2\le \infty$ with $R_2/ R_1 \ge R_0$, and $ u \in \E_{\ell, m}$ be such that $E( u; R_1, R_2)  < \de$. Then, there exists $\ell_0 \in \Z$  such that $| u(r) - \ell_0 \pi|<\eps$ for almost all $r \in (R_1, R_2)$. 

Moreover, there exist  constants $C=C(R_0), \al= \al(R_0)>0$ such that if $E(u; R_1, R_2)< \al$,  then 
\EQ{ \label{eq:H-E-comp} 
 \| u - \ell_0 \pi \|_{\E(R_1, R_2)} \le C E(  u; R_1, R_2). 
 }
\end{lem} 
\begin{proof} 
By an approximation argument we can assume $u \in \E_{\ell, m}$ is smooth. First, we show that for any $\eps_0>0$, there exists $r_0 \in [R_1, R_2]$ such that $| u(r_0) - \ell_0 \pi|< \eps_0$ for some $\ell_0 \in \Z$ as long as $E( (u; R_1, R_2)$ is sufficiently small. 
If not, one could find $\eps_1>0$,  $0< R_1< R_2$, and a sequence $u_n \in \E_{\ell, m}$ so that $E( u_n; R_1; R_2) \to 0$ as $n \to \infty$ but such that $\inf_{r \in[R_1, R_2], \ell \in \Z}   | u_n(r) - \ell \pi| \ge \eps_1$. The latter condition gives a constant $c( \eps_1)>0$ such that $\inf_{r \in [R_1, R_2]} |\sin( u_n(r))| \ge c(\eps_1)$. But then
\EQ{
E( u_n; R_1; R_2)  \ge \frac{k^2}{2}  \int_{R_1}^{R_2} \sin^2( u_n(r)) \,  \frac{\ud r}{r} \ge \frac{k^2}{2} c(\eps_1)^2 \log (R_2/R_1), 
}
which is a contradiction. Next define the function, 
$
G(u) = \int_0^u \abs{ \sin \rho} \, \ud \rho, 
$
and for $r_1 \in (R_1, R_2)$  note the inequality, 
\EQ{
\abs{G(u(r_0)) - G( u(r_1))} = \Big|\int_{u(r_1)}^{u(r_0)}  \abs{ \sin \rho} \, \ud \rho \Big| = \Big|\int_{r_1}^{r_0} \abs{ \sin u(r) } \abs{ \p_r u(r)} \, \ud r \Big| \lesssim E(u; R_1, R_2). 
}
We conclude using that $G$ is continuous and increasing that $| u(r) - \ell_0 \pi|<\eps$ for all $r \in (R_1, R_2)$. As long as $\eps>0$ is small enough we see that in fact, $\sin^2(u(r)) \ge \frac{1}{2}| u(r) - \ell_0 \pi|^2$ for all $r \in (R_1, R_2)$ and~\eqref{eq:H-E-comp} follows. 
\end{proof}

Given a mapping $u: (0, \infty) \to \R$ we define its energy density, 
\EQ{
 \bfe(u(r), r):= \frac{1}{2} \Big( (\p_r u(r))^2 + \frac{k^2}{r^2} \sin^2(u(r)) \Big) . 
}

\begin{lem}[Localized energy inequality]  Let $ I \subset  [0, \infty)$ be a time interval, and let $\phi: I \times (0, \infty) \to [0, \infty)$ be a smooth function. Let $u(t)\in \E_{\ell, m}$ be a solution to~\eqref{eq:hf} on $I$. Then, for any $t_1<t_2 \in I$, 
\EQ{ \label{eq:local-energy} 
&\int_{t_1}^{t_2} \int_0^\infty (\p_t u (t, r))^2  \phi(t, r)^2 \, r \, \ud r \ud t + \int_0^\infty \bfe(u(t_2, r), r) \phi(t_2, r)^2 \, r \, \ud r  \\
& =\int_0^\infty  \bfe(u(t_1, r), r) \phi(t_1, r)^2 \, r\, \ud r - 2\int_{t_1}^{t_2} \int_0^\infty \p_t u(t, r) \p_r u(t, r) \phi(t, r) \p_r \phi(t, r)\,r \, \ud r \ud t  \\
&\quad + 2\int_{t_1}^{t_2} \int_0^\infty \bfe(u(t, r), r)\phi(t, r) \p_t \phi(t, r) \, r \, \ud r \, \ud t
}
If $\phi(t, r)$ satisfies, $\p_t \phi(t, r) \le 0$ for all $t \in [t_1, t_2]$ then, 
\EQ{ \label{eq:loc-en-ineq-1} 
\int_0^\infty \bfe(u(t_2), r)& \phi(t_2, r)^2 \, r \, \ud r  + \frac{1}{2} \int_{t_1}^{t_2} \int_0^\infty (\p_t u (t, r))^2  \phi(t, r)^2 \, r \, \ud r \ud t \\
& \le \int_0^\infty  \bfe(u(t_1), r) \phi(t_1, r)^2 \, r\, \ud r  + 2\int_{t_1}^{t_2}\int_0^\infty (\p_r u(t, r))^2 (\p_r \phi(t, r))^2 \, r \,\ud r \ud t, 
}
and, 
\EQ{ \label{eq:loc-en-ineq} 
\int_0^\infty \bfe(u(t_2), r)& \phi(t_2, r)^2 \, r \, \ud r  + \int_{t_1}^{t_2} \int_0^\infty (\p_t u (t, r))^2  \phi(t, r)^2 \, r \, \ud r \ud t \\
& \le \int_0^\infty  \bfe(u(t_1), r) \phi(t_1, r)^2 \, r\, \ud r \\
&\quad  + 2\sqrt{E(u(t_1))}(t_2 - t_1)^{\frac{1}{2}} \Big(\int_{t_1}^{t_2}\int_0^\infty (\p_t u(t, r))^2 (\p_r \phi(t, r))^2(\phi(t, r))^2 \, r \,\ud r \ud t \Big)^{\frac{1}{2}}. 
}

\end{lem} 

\begin{proof} 
By an approximation argument we may assume that $u$ is smooth. Then~\eqref{eq:local-energy} is obtained for smooth solutions to~\eqref{eq:hf} by multiplying the equation by $\p_t u \phi^2$ and integrating by parts. The subsequent inequalities follow from Cauchy-Schwarz. 
\end{proof} 


\subsection{Profile decomposition} 

We state a profile decomposition in the sense of G\'erard~\cite{Gerard}, adapted to sequences of functions in the affine spaces $\E_{\ell, m}$; see also~\cite{BrezisCoron, Lions1, Lions2,  MeVe98, BG}. We use the analysis of sequences in $\E_{\ell, m}$ by Jia and Kenig in~\cite{JK}, which synthesized C\^ote's analysis in~\cite{Cote15}. 

\begin{lem}[Linear profile decomposition]  \label{lem:pd}  Let $\ell, m \in\Z$ and let $ u_n$ be a sequence in $\calE_{\ell, m}$ with $\limsup_{n \to \infty} E( u_n) <\infty$. Then, there exists $K_0 \in \{0, 1, 2, \dots\}$, sequences $ \lam_{n, j} \in (0, \infty)$ for $j \in \{1, \dots, K_0\}$,  $\sigma_{n, i} \in (0, \infty)$ for $i \in \N$, as well as mappings $ \psi^j \in \E_{\ell_j, m_j}$ with $E( \psi^j) < \infty$, and mappings  $ v\lin^{i} \in\E_{0, 0}$ such that for each $J \ge 1$, 
\EQ{
 u_n &= m  \pi +  \sum_{j = 1}^{K_0} ( \psi^j \big(  \frac{ \cdot}{\lam_{n, j}} \big) - m_j  \pi)   + \sum_{i =1}^J  v^i \big(  \frac{ \cdot}{\s_{n, i}} \big) + w_{n}^J( \cdot)
}
so that, 
\begin{itemize} 
\item  the parameters $\lam_{n, j}$ satisfy
\EQ{
\lam_{n, 1} \ll \lam_{n, 2} \ll \dots  \ll  \lam_{n, K_0} \mas n \to \infty; 
}
and for each $j$  one of  $\lam_{n, j} \to 0$, $\lam_{n, j} = 1$ for all $n$, or $\lam_{n, j} \to \infty$ as $n \to \infty$, holds;    
\item for each $i$ either $\s_{n, i} \to 0$, $\s_{n, i} = 1$ for all $n$, or $\s_{n, i} \to \infty$ as $n \to \infty$; 
\item for each $i \in \N$, 
\EQ{
\frac{\lam_{n, j}}{\s_{n, i}} + \frac{\s_{n, i}}{\lam_{n, j}}   \to \infty \mas n \to \infty \quad \forall j = 1, \dots, K_0; 
}
\item the scales $\s_{n, i}$  satisfy, 
\EQ{
\frac{\s_{n, i}}{\s_{n, i'}} + \frac{ \s_{n, i'}}{\s_{n, i}}  \to \infty \mas n \to \infty; 
}
\item the integers $\ell_j$ and $m_j$ satisfy, $\abs{\ell_j - m_j} \ge 1$, and, 
\EQ{
\ell = m + \sum_{j=1}^{K_0} ( \ell_j - m_j) ; 
}
\item the error term $\bs w_{n}^J$ satisfies, 
\EQ{
& w_{n}^J( \lam_{n, j} \cdot) \rightharpoonup 0 \in \E \mas n \to \infty\\
& w_{n }^J(  \s_{n, i}  \cdot) \rightharpoonup 0\in \E \mas n \to \infty
}
for each $J \ge 1$, each $j = 1, \dots, K_0$, and $i \in \N$,  and vanishes strongly in the sense that 
\EQ{ \label{eq:w-to-0} 
\lim_{J \to \infty} \limsup_{n \to \infty}   \| w_{n}^J \|_{L^\infty}   = 0; 
}
\item the following  pythagorean decomposition of the nonlinear energy holds:  for each $J \ge 1$, 
\EQ{ \label{eq:pyth} 
E( u_n) &= \sum_{j =1}^{K_0} E( \psi^j)  + \sum_{i =1}^J E( v^j)   + E( w_{n}^J)  + o_n(1) 
} 
as $n \to \infty$. 
 
\end{itemize} 

\end{lem} 

\begin{proof}[Sketch of Proof]
We follow Jia and Kenig's argument~\cite[Proof of Lemma 5.5]{JK} to first extract the profiles $\psi^j \in \E_{\ell_j, m_j}$ at the scales $\lam_{n, j}$, see~\cite[Pages 1594-1600]{JK}. Since these all have energy $\ge E( Q)$, there can only be finitely many of them, which defines the non-negative integer $K_0$. The conclusion of their argument yields a sequence, 
\EQ{
h_n := u_n - m \pi - \sum_{j= 1}^{K_0} (\psi^j_{\lam_{n, j}} - m_j \pi)   \in \E_{0, 0}
}
with $\limsup_{n \to \infty} \| h_n \|_{\E} < \infty$. Setting $H_n:= r^{-k} h_n$ we see that $\limsup_{n \to \infty} \| H_n \|_{\dot H^{1}(\R^d)} < \infty$ for $d = 2k + 2$ (here we view $H_n$ as a sequence of radially symmetric functions on $\R^d$).  Thus we may apply G\'erard's profile decomposition~\cite[Theorem 1.1]{Gerard} for sequences in $\dot{H}^1(\R^d)$ to the sequence $H_n$ obtaining sequences of scales $\sigma_{n, i}$ and profiles $V^i$ so that for $W_n^J$ defined by 
\EQ{
H_n = \sum_{i=1}^J \sigma_{n, i}^{-\frac{d}{p^*}} V\Big( \frac{\cdot}{\sigma_{n, i}} \Big)  + W_{n}^J 
}
we have 
\EQ{
\lim_{J \to \infty}  \limsup_{n \to \infty} \| W_{n}^J \|_{L^{p^*}} = 0, 
}
along with the usual orthogonality of the scales and the pythagorean expansion of the $\dot H^1$ norm.  Note that here $p^*:= \frac{ 2d}{d-2}$ is the critical Sobolev exponent. We set $v^i(r) := r^k V^i(r)$ and $w_n^J(r) := r^{k} W_{n}^J(r)$ for each $i, n, J$. Note that $w_n^J \in \E$ and 
\EQ{ \label{eq:wp*} 
\lim_{J \to \infty}  \limsup_{n \to \infty} \int_0^\infty (w_n^J(r))^{p^*}  \, \frac{\ud r}{r}  = 0 .
}
We conclude by observing the inequality 
\EQ{
\sup_{r>0} \abs{ w(r)}^{\frac{p^*}{2} +1} \le C(p^*)  \Big(  \int_0^\infty (w(r))^{p^*}  \, \frac{\ud r}{r} \Big)^{\frac{1}{2}}  \Big(  \int_0^\infty (\p_r w(r))^{2}  \, r \, \ud r \Big)^{\frac{1}{2}}, 
}
which holds for all $w \in \E$. Thus~\eqref{eq:wp*} combined with the above gives the vanishing of the error as in~\eqref{eq:w-to-0}. 
\end{proof} 


\subsection{Multi-bubble configurations}
We study properties of finite energy maps near a multi-bubble configuration as in Definition~\ref{def:multi-bubble}. We record here several lemmas proved in~\cite{JL6}. 

The operator $\LL_{\calQ}$ obtained by linearization of~\eqref{eq:hf} about an $M$-bubble configuration $ \calQ(m, \vec \iota, \vec \lam)$ is given by, 
\EQ{  \label{eq:LQ-def} 
\LL_{\calQ} \, g := \uD^2 E(\calQ(m, \vec\iota, \vec \lam)) g = - \p_r^2 g - \frac{1}{r} \p_r g + \frac{k^2}{r^2} f'(\calQ(m, \vec\iota, \vec \lam) )g, 
}
where $f'( z) = \cos 2 z$. 
 Given $g \in \E$, 
\EQ{
\La \uD^2 E(\calQ(m, \vec \iota, \vec \lam))  g \mid  g \Ra =  \int_0^\infty \Big( (\p_r g(r))^2 + \frac{k^2}{r^2} f'(\calQ(m, \vec \iota, \vec \lam)) g(r)^2 \, \Big) r \ud r. 
}
An important instance of the operator $\LL_{\calQ}$ is given by linearizing \eqref{eq:hf} about a single harmonic map $ \calQ(m, M, \vec\iota, \vec \lam) = Q_{\lam}$. In this case we use the short-hand notation, 
\EQ{ \label{eq:LL-def} 
\LL_{\lam} := (-\De + \frac{k^2}{r^2}) + \frac{k^2}{r^2} ( f'(Q_{\lam}) - 1) 
}
We write $\calL := \calL_1$.
For each $k \ge 1$,  
\EQ{
\Lam Q(r):= r \p_r Q(r)  = k \sin Q = 2 k  \frac{ r^k}{1 + r^{2k}}
}
 When $k \ge 2$, $\Lam Q$ is a zero energy eigenfunction for $\calL$, i.e.,  
\EQ{
\calL \Lam Q = 0, \mand \Lambda Q  \in L^2_{\textrm{rad}}(\R^2).
}
 When $k=1$, $\calL \Lam Q = 0$ holds but  $\Lam Q \not \in L^2$ due to slow decay as $r \to \infty$  and $0$ is called a threshold resonance. 

 
We define a smooth non-negative function $\calZ \in C^{\infty}(0, \infty) \cap L^1((0, \infty), r\, \ud r)$ by  
 \EQ{ \label{eq:Z-def} 
 \calZ(r) := \begin{cases}   \chi(r) \Lam Q(r) \mif k =1, 2 \\ \Lam Q(r)  \mif k \ge 3 \end{cases}
 }
and note that  
 \EQ{
 \ang{ \calZ \mid \Lam Q} >0  \label{eq:ZQ} . 
 }
The precise form of $\calZ$ is not so important, rather only that it is not perpendicular to $\Lam Q$ and has sufficient decay and regularity. We fix it as above because of the convenience of setting $\calZ = \Lam Q$  if $k\ge 3$. 
We record the following localized coercivity lemma proved in~\cite{JJ-AJM}.

\begin{lem}[Localized coercivity for $\LL$] \emph{\cite[Lemma 5.4]{JJ-AJM}}  \label{l:loc-coerce} 
Fix $k \ge 1$.  There exist uniform constants $c< 1/2, C>0$ with the following properties. Let $g \in \E$. Then, 
\EQ{ \label{eq:L-coerce}
\ang{ \LL g \mid g} \ge c  \| g \|_{H}^2  - C\ang{ \calZ \mid g}^2 
}
If $R>0$ is large enough then,  
\EQ{ \label{eq:L-loc-R} 
(1-2c)&\int_0^{R} \Big((\p_r g)^2 + k^2 \frac{g^2}{r^2} \Big) \, r \ud r +  c \int_{R}^\infty \Big((\p_r g)^2 + k^2 \frac{g^2}{r^2} \Big) \, r \ud r  + \La\frac{ k^2}{r^2}(f'(Q) - 1) g\mid g\Ra \\
& \ge  - C\ang{ \calZ \mid g}^2. 
}
If $r>0$ is small enough, then
\EQ{ \label{eq:L-loc-r} 
(1-2c)&\int_r^{\infty} \Big((\p_r g)^2 + k^2 \frac{g^2}{r^2} \Big) \, r \ud r +  c \int_{0}^r \Big((\p_r g)^2 + k^2 \frac{g^2}{r^2} \Big) \, r \ud r  + \La\frac{ k^2}{r^2}(f'(Q) - 1) g\mid g\Ra \\
& \ge  - C\ang{ \calZ \mid g}^2. 
}
\end{lem} 

As a consequence, (see for example~\cite[Proof of Lemma 2.4]{JKL1} for an analogous argument) one obtains the following coercivity property of the operator $\LL_{\calQ}$. 

\begin{lem} \label{lem:D2E-coerce} \emph{\cite[Lemma 2.19]{JL6}} Fix $k \ge 1$, $M \in \N$. There exist $\eta, c_0>0$ with the following properties. Consider the subset of $M$-bubble configurations $\calQ(m, \vec\iota, \vec \lam)$ for $\vec \iota \in \{-1, 1\}^M$, $\vec \lam \in (0, \infty)^M$ such that, 
\EQ{ \label{eq:lam-ratio} 
\sum_{j =1}^{M-1} \Big( \frac{\lam_j}{\lam_{j+1}} \Big)^k \le \eta^2. 
}
Let $g \in H$ be such that 
\EQ{
0 = \ang{ \calZ_{\U{\lam_j}} \mid g}  \mfor j = 1, \dots M. 
}
for some $\vec \lam$ as in~\eqref{eq:lam-ratio}. Then, 
\EQ{
\ang{ \uD^2 E( \calQ( m, \vec \iota, \vec \lam)) g \mid g} \ge c_0 \| g \|_{\E}^2. 
} 
\end{lem} 

The following technical lemma is useful when computing interactions between bubbles at different scales. 
\begin{lem}
\label{lem:cross-term}
For any $\lambda \leq \mu$ and $\alpha, \beta > 0$ with $\alpha \neq \beta$ the following bound holds:
\begin{equation}
\int_0^{\infty} \max\Big(1, \frac{r}{\lambda}\Big)^{-\alpha}\max\Big(1, \frac{\mu}{r}\Big)^{-\beta} \frac{\vd r}{r}
\lesssim_{\alpha, \beta} \Big(\frac{\lambda}{\mu}\Big)^{\min(\alpha, \beta)}.
\end{equation}
For any $\alpha > 0$ the following bound holds:
\begin{equation}
\int_0^\infty\max\Big(1, \frac{r}{\lambda}\Big)^{-\alpha}\max\Big(1, \frac{\mu}{r}\Big)^{-\alpha} \frac{\vd r}{r}
\lesssim_{\alpha} \Big(\frac{\lambda}{\mu}\Big)^{\alpha}\Big( 1+ \log\Big(\frac{\mu}{\lambda}\Big)\Big).
\end{equation}
\end{lem}
\begin{proof}
This is a straightforward computation, considering separately the regions $0 < r \leq \lambda$, $\lambda \leq r \leq \mu$, and $r \geq \mu$.
\end{proof}

Using the above, along with the formula for $\calZ$ in~\eqref{eq:Z-def} we obtain the following. 
\begin{cor}  \label{cor:ZQ} 
Let $\calZ$ be as in~\eqref{eq:Z-def} and suppose that $\lam, \mu>0$ satisfy $\lam/ \mu \le 1$. Then, 
\EQ{
\ang{ \calZ_{\U \lam} \mid \Lam Q_{\U \mu}}  \lesssim  \begin{cases} (\lam/\mu)^{k+1} \mif k=1, 2 \\ (\lam/ \mu)^{k-1} \mif k \ge 3 \end{cases} ,  \quad \ang{ \calZ_{\U \mu} \mid \Lam Q_{\U \lam}}  \lesssim \begin{cases} 1 \mif k =1 \\  (\lam/\mu)^{k-1} \mif  k \ge 2\end{cases} 
}
\end{cor}

 Lemma~\ref{lem:cross-term} is also used to prove the following lemma from~\cite{JL6} giving leading order terms in an expansion of the nonlinear energy functional about an $M$-bubble configuration. We refer the reader to~\cite{JL6} for the proof. 

\begin{lem}\emph{\cite[Lemma 2.22]{JL6} } \label{lem:M-bub-energy} Fix $k\ge1,  M \in \N$. 
For any $\te>0$, there exists $\eta>0$ with the following property. Consider the subset of $M$-bubble $ \calQ(m,\iota, \vec \lam)$ configurations 
such that 
\EQ{
\sum_{j =1}^{M-1} \Big( \frac{ \lam_{j}}{\lam_{j+1}} \Big)^k \le \eta.
}
Then, 
\EQ{ \label{eq:mb-energy} 
  \Big|  E( \calQ( m, \vec \iota, \vec \lam))  - M E(  Q) -  16 k \pi \sum_{j =1}^{M-1} \iota_j \iota_{j+1}  \Big( \frac{ \lam_{j}}{\lam_{j+1}} \Big)^k  \Big| \le \te \sum_{j =1}^{M-1} \Big( \frac{ \lam_{j}}{\lam_{j+1}} \Big)^k .
}
Moreover, there exists a uniform constant $C>0$ such that for any $g \in H$, 
\EQ{\label{eq:mb-linear}
\abs{\ang{ \uD E( \calQ(m, \vec \iota, \vec \lam)) \mid g} } \le C \| g \|_{\E} \sum_{j =1}^M \Big( \frac{\lam_{j}}{\lam_{j+1}} \Big)^k . 
}
\end{lem} 

The following (standard) modulation lemma plays an important role and we refer the reader to~\cite[Lemma 2.25]{JL6} for its proof. Before stating it, we define a proximity function to $M$-bubble configurations. Fixing $m, M$ we observe that $\calQ(m, \vec\iota, \vec\lambda; r)$ is an element of $\E_{\ell, m}$, where 
\EQ{ \label{eq:ell-def} 
\ell = \ell(m, M, \vec \iota): = m - \sum_{j=1}^M  \iota_j
}

 \begin{defn} \label{def-d} Fix $m, M$ as in Definition~\ref{def:multi-bubble} and let $ v \in \E_{\ell, m}$ for some $\ell \in \Z$.  Define, 
\EQ{ \label{eq:d-def} 
\bfd(  v) = \bfd_{ m, M}(  v) := \inf_{\vec \iota, \vec \lam}  \Big( \|  v -  \calQ( m, \vec \iota, \vec \lam) \|_{\E}^2 + \sum_{j =1}^{M-1} \Big( \frac{\lam_{j}}{\lam_{j+1}} \Big)^k \Big)^{\frac{1}{2}}.
}
where the infimum is taken over all vectors $\vec \lam = (\lam_1, \dots, \lam_M) \in (0, \infty)^M$ and all $\vec \iota = \{ \iota_1, \dots, \iota_M\} \in \{-1, 1\}^M$ satisfying~\eqref{eq:ell-def}. 
\end{defn}


\begin{lem}[Static modulation lemma] \label{lem:mod-static}\emph{ \cite[Lemma 2.25]{JL6}} Fix $k \ge 1$ and $M \in \N$. 
There exists $\eta \in (0, 1)$, $C>0$ with the following properties.  
Let  $m$ be as in Definition~\ref{def:multi-bubble} and $\bfd_{m, M}$ as in Definition~\ref{def-d}. Let $\te>0$, $\ell \in \Z$,  and let $ v \in  \calE_{\ell,  m}$   be such that 
\EQ{ \label{eq:v-M-bub} 
\bfd_{ m, M}(  v)  \le \eta, \mand E(  v) \le ME(  Q) + \te^2, 
}
Then, there exists a unique choice of $\vec \lam = ( \lam_1, \dots, \lam_M) \in  (0, \infty)^M$, $\vec\iota \in \{-1, 1\}^M$, and $g \in  H$, such that  
\EQ{ \label{eq:v-decomp} 
    v &=   \calQ( m, \vec \iota, \vec \lam) +  g, \\
   0 & = \La \calZ_{\U{\lam_j}} \mid g\Ra , \quad \forall j = 1, \dots, M,
   }
   along with the estimates, 
\EQ{  \label{eq:g-bound-0}
\bfd_{ m, M}(  v)^2 &\le \|   g \|_{\E}^2  + \sum_{j =1}^{M-1} \Big( \frac{\lam_{j}}{\lam_{j+1}} \Big)^k  \le C \bfd_{ m, M}(  v)^2,
}
and, 
\EQ{\label{eq:g-bound-A} 
   \|   g \|_{\E}^2 + \sum_{j \not \in \calA}   \Big( \frac{ \lam_j}{ \lam_{j+1} }\Big)^k & \le C  \max_{ j \in \calA} \Big( \frac{ \lam_j}{ \lam_{j+1} }\Big)^k + \te^2, 
}
where $\calA  := \{ j \in \{ 1, \dots, M-1\} \, : \, \iota_j  \neq \iota_{j+1} \}$. 
\end{lem}

We also use of the following lemma proved from~\cite{JL6} which says that a finite energy map cannot be close to two distinct multi-bubble configurations. 

\begin{lem}  \label{lem:bub-config} \emph{\cite[Lemma 2.27]{JL6}} Let $k \ge 1$. 
 There exists $\eta>0$ sufficiently small with the following property. Let $m, \ell \in \Z$, $M, L \in \N$,  $\vec\iota \in \{-1, 1\}^M, \vec \sigma \in \{-1, 1\}^L$, $\vec \lam \in (0, \infty)^M, \vec \mu \in (0, \infty)^L$,   and $w$ be such that $E_{\bfp}( w) < \infty$ and, 
 \begin{align} 
 \|w  - \calQ(m,  \vec \iota, \vec \lam)\|_{\E}^2  + \sum_{j =1}^{M-1} \Big(\frac{\lam_j}{\lam_{j+1}} \Big)^{k} &\le \eta,  \label{eq:M-bub} \\
 \|w  - \calQ(\ell, \vec \sigma , \vec \mu)\|_{\E}^2 +  \sum_{j =1}^{L-1} \Big(\frac{\mu_j}{\mu_{j+1}} \Big)^{k} &\le \eta. \label{eq:L-bub} 
 \end{align} 
 Then, $m = \ell$, $M = L$, $\vec \iota = \vec \sigma$. Moreover, for every $\te>0$ the number $\eta>0$ above can be chosen small enough so that 
 \EQ{ \label{eq:lam-mu-close} 
\max_{j = 1, \dots M} | \frac{\lam_j}{\mu_j} - 1 | \le  \te.
 }
\end{lem}

\section{Localized sequential bubbling} 
We define a localized distance function 
\EQ{ \label{eq:delta-def} 
\bs\delta_{R}(u) := \inf_{m, M, \vec \iota, \vec \lam} \Big( \| u - \calQ(m, \vec \iota, \vec \lam) \|_{\E(r \le R)}^2 + \sum_{j =1}^M \Big( \frac{\lam_j}{\lam_{j+1}} \Big)^k \Big)^{\frac{1}{2}}
}
where the infimum is taken over all $m \in \Z$, $M \in \{0, 1, 2, \dots\}$, all vectors $\iota \in\{-1, 1\}^M$, $\vec \lam \in (0, \infty)^M$, and  we use the convention that the last scale $\lam_{M+1} = R$. 
%

\begin{lem} \label{lem:compact} 
Let $\ell, m \in \Z$ and let $u_n \in \E_{\ell, m}$ be a sequence of maps with $\limsup_{n \to \infty} E( u_n) < \infty$. Let $\rho_n \in (0, \infty)$ be a sequence and suppose that 
\EQ{ \label{eq:no-tension} 
\lim_{n \to \infty} (\rho_n \| \calT(u_n) \|_{L^2}) = 0.
}
Then, there exists a sequence $R_n \to \infty$ so that, up to passing to a subsequence of the $u_n$, we have, 
\EQ{
\lim_{n \to \infty} \bs \de_{R_n\rho_n}( u_n)  = 0.
}
The subsequence of the $u_n$ can be chosen so that there are fixed  $(M, m, \vec \iota) \in  \N \cup\{0\} \times  \Z \times \{-1, 1\}^M$, a sequence  $\vec \lam_n \in (0, \infty)^M$, and $C_0 >0$ with 
\EQ{
\lim_{n \to \infty} \Big( \| u_n - \calQ(m, \vec \iota, \vec \lam_n) \|_{\E(r \le  R_n \rho_n)}^2 + \sum_{j =1}^{M-1} \Big( \frac{\lam_{n,j}}{\lam_{n, j+1}} \Big)^k\Big) = 0,  
}
and, 
\EQ{
\lam_{n ,M} \le C_0\rho_n, \quad \forall\,  n. 
} 
\end{lem} 

\begin{rem} 
Lemma~\ref{lem:compact} is proved in the general (non-equivariant) setting by Qing~\cite{Qing}. Here we give a different (but related) treatment adapted to the equivariant setting using explicitly the notion of a profile decomposition of G\'erard~\cite{Gerard}. The proof that no energy can accumulate in the ``neck'' regions between the bubbles can be simplified in the equivariant setting and here we use an argument due to Jia and Kenig~\cite{JK} from their proof of an analogous result for equivariant wave maps; see Lemma~\ref{lem:jk} below. 
\end{rem} 

\begin{lem}
\label{lem:sequences}
If $a_{k, n}$ are positive numbers such that $\lim_{n\to \infty}a_{k, n} = \infty$ for all $k \in \bN$,
then there exists a sequence of positive numbers $b_n$ such that $\lim_{n\to \infty} b_n = \infty$
and $\lim_{n\to \infty} a_{k, n} / b_n = \infty$ for all $k \in \bN$.
\end{lem}
\begin{proof}
For each $k$ and each $n$ define $\ti a_{k, n} = \min\{ a_{1, n}, \dots, a_{k, n}\}$. Then the sequences $\ti a_{k, n} \to \infty$ as $n \to \infty$ for each $k$, but also satisfy $\ti a_{k, n} \le a_{k, n}$ for each $k, n$, as well as $\ti a_{j, n} \le \ti a_{k, n}$ if $j>k$. Next, choose a strictly increasing sequence $\{n_k \}_k \subset \N$ such that $\ti a_{k, n} \ge k^2$ as long as $n \ge n_k$. For $n$ large enough, let $b_n \in \bN$ be determined by the condition
$n_{b_n} \leq n < n_{b_n + 1}$. Observe that $b_n \to \infty$ as $n \to \infty$. Now fix any $\ell \in \N$ and let $n$ be such that $b_n > \ell$. We then have
\begin{equation}
a_{\ell, n} \geq \ti a_{\ell, n} \ge \ti a_{b_n, n} \ge  b_n^2  \gg b_n.
\end{equation}
Thus the sequence $b_n$ has the desired properties. 
\end{proof}

The proof of the Lemma~\ref{lem:compact} consists of several steps, which are designed to reduce the proof to a scenario already considered by C\^ote in~\cite[Proof of Lemma 3.5]{Cote15} and then by Jia-Kenig in~\cite[Proof of Theorem 3.2]{JK}, albeit in a different context. In particular, we will seek to apply the following result from~\cite{JK}.  

\begin{lem} \emph{\cite[Theorem 3.2]{JK}} \label{lem:jk} 
Let $v_n$ be a sequence of maps such that $\limsup_{n \to \infty} E( v_n) < \infty$. Suppose that there exists a sequence an integer $M \ge 0$  and scales $\lam_{n, 1} \ll  \dots \ll \lam_{n, M} \lesssim 1$ such that 
\EQ{
{v}_n &= m_1  \pi + \sum_{j =1}^{M} \iota_j(   Q \big(  \frac{ \cdot}{\lam_{n, j}} \big)   -  \pi)  + {w}_{n, 0}, 
}
where $\|{w}_{n} \|_{L^\infty} \to 0$ and $\| { w}_{n} \|_{\E(r \ge  r_n^{-1})} \to 0$ as $n \to \infty$ for some sequence $r_{n} \to \infty$. Suppose in addition that,   $\| {w}_{n} \|_{\E(A^{-1} \lam_n \le r \le A \lam_n)} \to 0$ as $n \to \infty$ for any sequence $\lam_n \lesssim 1$ and any $A>1$, 
and finally, that 
\EQ{ \label{eq:jk0} 
\limsup_{n\to \infty}\int_0^\infty \bigg( k^2 \frac{\sin^2(2 v_n)}{2r^2} + (\p_r v_n)^2 2 \cos (2 v_n)  \bigg) \,r \,  \ud r    \le 0 . 
}
Then, 
\EQ{
\| {w}_{n} \|_{\E} \to 0 \mas n \to \infty. 
}
\end{lem} 

\begin{rem} 
 Lemma~\ref{lem:jk} is not stated in~\cite{JK} exactly as given above. However, an examination of~\cite[Proof of Theorem 3.2]{JK} shows that this is precisely what is established.   The heart of the matter lies in the fact that the Jia-Kenig virial functional~\eqref{eq:jk0} vanishes at $Q$, i.e., 
\EQ{
\int_0^\infty \bigg( k^2 \frac{\sin^2(2Q)}{2r^2} + (\p_r Q)^2 2 \cos (2 Q)  \bigg) \,r \,  \ud r    = 0 , 
} 
but gives coercive control of the energy in regions where $v_n( r)$ is near integer multiples of~$\pi$.
\end{rem}

\begin{proof}[Proof of Lemma~\ref{lem:compact}] By rescaling we may assume that $\rho_n =1$ for each $n$. 

First, we observe that after passing to a subsequence,  $u_n$ admits a profile decomposition, 
 \EQ{
  u_n &= m  \pi +  \sum_{j = 1}^{K_0} ( \psi^j \big(  \frac{ \cdot}{\lam_{n, j}} \big) - m_j  \pi)   + \sum_{i =1}^J  v^i \big(  \frac{ \cdot}{\s_{n, i}} \big) + w_{n}^J( \cdot). 
 }
 where the profiles $(\psi^j, \lam_{n,j}), (v^j, \sigma_{n, j})$ and the error satisfy the conclusions of Lemma~\ref{lem:pd}. 

 \textbf{Step 1.} We make an initial restriction on the sequence $R_n \to \infty$, refining our choice of this sequence later in the proof. Consider the sets of indices 
 \EQ{
 \calJ_{\infty} := \{ j \in \{1, \dots, K_0\} \mid \lim_{n \to \infty} \lam_{n, j} = \infty\}, \quad \mathcal{I}_{\infty} := \{ i \in \N \mid \lim_{n \to \infty} \sigma_{n, i} = \infty\}
 }
 By Lemma~\ref{lem:sequences} we choose a sequence $R_{n,1} \to \infty$ so that $R_{n, 1} \ll \lam_{n, j} , \sigma_{n, i}$ for each $\lam_{n, j}$ with $j \in \calJ_{\infty}$ and each $\sigma_{n, i}$ with $i \in \calI_{\infty}$. If follows that
 \EQ{
\lim_{n \to \infty}  E( \psi^j( \cdot/ \lam_{n, j}); 0, R_{n, 1} )  = 0, \quad \lim_{n \to \infty}  E( v^j( \cdot/ \sigma_{n, i}); 0, R_{n, 1} )  = 0
 }
 for any of the indices $j \in \calJ_{\infty}$ or $i \in \calI_{\infty}$, and thus these profiles do not factor into the distance $\bs \delta_{R_{n}} (u_n)$ for any sequence $R_n \le R_{n, 1}$. 
 
 \textbf{Step 2:} Next we perform a bubbling analysis on the profiles with bounded scale. Define
 \EQ{
  \calJ_{0} := \{ j \in \{1, \dots, K_0\} \mid \lim_{n \to \infty} \lam_{n, j} < \infty\}, \quad \mathcal{I}_{0} := \{ i \in \N \mid \lim_{n \to \infty} \sigma_{n, i} < \infty\}
   }
   First, for $j \in \calJ_0$ and $i \in \calI_0$, denote 
   \EQ{
   u_n^j(r) := u_n( \lam_{n, j} r), \quad u_n^i(r):= u_n( \sigma_{n, i} r) 
   }
   Then we have $u_n^j \to \psi^j$ as $n \to \infty$ locally uniformly in $(0, \infty)$ and weakly in $\dot H^1(\R^2)$ (that is, if we view each $u_n^j$ as a radially symmetric function on $\R^2$).  These convergence properties are by construction, see~\cite[pg. 1594]{JK}). Moreover, since $\lim_{n  \to \infty} \lam_{n, j} < \infty$ we have, 
   \EQ{
   \| \calT( u_n^j) \|_{L^2} = \lam_{n, j}  \| \calT(u_n) \|_{L^2} \to 0 \mas n \to \infty
   }
   It follows that, 
   \EQ{
   \La \calT( \psi^j) \mid \phi \Ra_{L^2} = 0 
   }
   for all $ \phi \in C^\infty_0(0, \infty)$, i.e., $\psi^j$ is a weak harmonic map, and hence a smooth harmonic map by H\'elein~\cite{Hel}. Since $\abs{m_j - \ell_j} \ge 1$ we see that $E(\psi^j) \ge E( Q)$, and thus $\psi^j = \ell_j \pi + \iota_j Q_{\lam_{j, 0}}$ for some $\iota_j \in \{-1, 1\}$ and some fixed scale $\lam_{j, 0}$ and $m_j = \ell_j + \iota \pi$.  We will abuse notation and replace $\lam_{n, j}$ with  $\lam_{n, j} \lam_{n, 0}$  while still calling this sequence $\lam_{n, j}$. 
   
 We perform the same analysis with the $u_n^i$  and $v^i$, concluding that each $v^i$ is a smooth harmonic map. But since  $v^i \in \E_{0, 0}$ we find that $v^i  \equiv 0$ for every $i \in \calI_0$. 

\textbf{Step 3:} Next, by~\eqref{eq:no-tension} and recalling that we have rescaled so that $\rho_n = 1$,  we let $R_{2, n} \to \infty$ be a sequence such that 
\EQ{
1 \ll R_{2, n} \ll \| \calT( u_n ) \|_{L^2}^{-1} . 
}
Then, by Cauchy-Schwarz
\EQ{ \label{eq:jk}
\Big|\La \calT(u_n)  \mid \sin(2 u_n) \chi_{\ti R_{n}} \Ra \Big| \le  \| \calT( u_n ) \|_{L^2} \ti R_{ n}  \to 0 \mas n \to \infty
}
for any sequence $\ti R_{n} \le R_{2, n}$. We define $R_{3, n} := \min( R_{1, n}, R_{2, n})$. 

\textbf{Step 4:} We close in on the final selection of the sequence $R_n$, choosing first $\sqrt{R_{3, n}} \le R_{4, n} \le (1/2) R_{3, n}$ so that 
\EQ{
E( u_n; \frac{1}{4} R_n, 4R_n) \to 0 \mas n \to \infty
}
The existence of such a sequence is proved by pigeonholing; see for example~\cite[Eq. (3.12)]{JL6}. Using Lemma~\ref{lem:pi} we can, after passing to a subsequence, find an integer $m_1 \in \Z$ so that $|u_n(r) - m_1\pi| \to 0$ for a.e., $r \in [\frac{1}{4} R_n, 4 R_n]$, and we define a truncated sequence 
\EQ{
\ti u_n := \chi_{R_{4, n}} u_n + (1- \chi_{R_{4, n}}) m_1 \pi 
}
By construction we have the following decomposition for $\ti u_n$, 
\EQ{
\ti u_n  = m_1 \pi + \sum_{j \in \calJ_0} ( \iota_j Q_{\lam_j} -  \pi) + \ti w_n
}
where the error $\ti w_n := \chi_{R_{4,n}} w_{n}^J + o_n(1)$  (note we can drop the index $J$ since any nontrivial profiles from the index sets $\calJ_\infty$ or $\calI_\infty$ contribute a vanishing error in the region $r \le R_{4, n}$ by Step 1 and there are no nontrivial profiles from the index set $\calI_{0}$ from Step 2). We define $M:= \# \calJ_0$ and we reorder/relabel the profiles so that $\lam_{n, 1} \ll \lam_{n, 2} \ll  \dots \lam_{n, M}$ for the indices $j \in \calJ_0$.   Note that we have proved that 
\EQ{ \label{eq:Linfty-w} 
\lim_{n \to \infty} \| \ti w_n  \|_{L^\infty}  = 0
}
After passing to a subsequence of the $ u_n$, we claim there is a sequence $R_n \to \infty$ with the  properties, 
\EQ{\label{eq:rn} 
1 \ll R_n \le R_{4, n} , \quad \|{\ti w}_{n}\|_{\E( \frac{1}{4} R_n^{-1} \le r \le 4 R_n)} \to 0 \mas n \to \infty.
}
The existence of such a sequence is a consequence of the following property about $\ti w_{n}$: for any  
sequence $\lam_n \lesssim 1$ and any $A>1$ we have, 
\EQ{ \label{eq:every-scale} 
\| w_{n}\|_{\E(  \lam_n A^{-1} \le r \le  \lam_n A)} \to 0 \mas n \to \infty. 
}
The property~\eqref{eq:every-scale} was proved in~\cite[Step 2., p.1973-1975, Proof of Theorem 3.5]{Cote15} and ~\cite[Proof of (5.29) in Theorem 5.1]{JK}   and we refer the reader to those works for details of the argument, which also applies in the current setting. The intuition is that at any scale $\lam_{n} \lesssim 1$ at which $\ti u_n$ carries energy we have already extracted a profile $Q_{\lam_{n, j}}$ with $\lam_{n, j} \simeq \lam_n$. 
To prove~\eqref{eq:rn} we  consider the case $\lam_n =1$ in~\eqref{eq:every-scale}, and passing to a subsequence of the ${\ti u}_n$, we obtain a sequence as in~\eqref{eq:rn}. 

We truncate to the region $r \le R_n$, following the same procedure used to define $\bs {\ti u}_n$, using now $R_n$ in place of $R_{4, n}$. Indeed, set
\EQ{
{\breve{u}}_n(t_n, r) := \chi_{R_n}(r) {\ti u}_n(t, r) + (1-  \chi_{R_n}(r) )  m_1 \pi.
}
Defining ${\breve w}_{n, 0} := \chi_{R_n}(r) {\ti{w}}_{n}  + ( \chi_{R_n}(r) -1) \sum_{j =1}^{M}  \iota_j(   Q \big(  \frac{ \cdot}{\lam_{n, j}} \big)  -  \pi) $ and using that $ \lam_{n, 1} \ll \dots \ll \lam_{n, M} \lesssim 1$ along with~\eqref{eq:Linfty-w} and~\eqref{eq:rn} we see that, 
\EQ{ \label{eq:titiu} 
&{\breve{ u}}_n(t_n) = m_1  \pi + \sum_{j =1}^{M} \iota_j(   Q \big(  \frac{ \cdot}{\lam_{n, j}} \big)  -   \pi)  + { \breve{ w}}_{n, 0}, \mand \\
&\lim_{n \to \infty}\Big(  \| {\breve w}_{n} \|_{\E(R_n^{-1} \le r < \infty)} + \| \breve w_{n} \|_{L^\infty}  \Big)  = 0.
}
Moreover, by~\eqref{eq:every-scale} we see that for any sequence $\lam_n \lesssim 1$ and any $A>1$ that, 
\EQ{\label{eq:every-scale-titi} 
\lim_{n \to \infty} \|{\breve w}_{n}\|_{\E(  \lam_n A^{-1} \le r \le  \lam_n A)} = 0.
}
Note that since $\breve u_n(r) = u_n(r)$ for $r \le R_n$, we deduce from~\eqref{eq:jk} that, 
\EQ{
\Big| \La \calT( \breve u_n) \mid \sin (2 \breve u_n) \chi_{\frac{1}{4} R_n} \Ra \Big|  \to 0 \mas n \to \infty 
}
We claim that 
\EQ{
\Big| \La \calT( \breve u_n) \mid \sin (2 \breve u_n)(1-  \chi_{\frac{1}{4} R_n}) \Ra \Big|  \to 0 \mas n \to \infty 
}
as well. To see this, note that by~\eqref{eq:titiu} 
\EQ{
\lim_{n \to \infty} E( \breve u_n; r_n, \infty)  = 0
}
for any sequence $r_n \to \infty$. And after integration by parts we deduce the bound, 
\EQ{
\Big| \La \calT( \breve u_n) \mid \sin (2 \breve u_n)(1-  \chi_{\frac{1}{4} R_n}) \Ra \Big|  \lesssim E(\breve u_n; 1/8 R_n, \infty) \to 0 \mas n \to \infty
}
Hence, 
\EQ{
\Big| \La \calT( \breve u_n) \mid \sin (2 \breve u_n) \Ra \Big|  \to 0 \mas n \to \infty 
}
Integrating by parts on the left hand side, we see that 
\EQ{
\lim_{n  \to \infty}  \int_0^\infty \bigg( k^2 \frac{\sin^2(2 \breve u_n)}{2r^2} + (\p_r \breve u_n)^2 2 \cos (2 \breve u_n)  \bigg) \,r \,  \ud r   = 0.
}
The sequence $\breve u_n$ then satisfies all the conditions of Lemma~\ref{lem:jk} and we conclude that $ \lim_{ n \to \infty} \| \breve w_n \|_{\E}  = 0$. Since $\breve u_n(r) = u_n(r)$ for $r \le R_n$ we conclude that $ \lim_{n \to \infty} \bs \de_{R_n} ( u_n) = 0$. An examination of the decomposition~\eqref{eq:titiu} yields the remaining claims in from Lemma~\ref{lem:compact}. 
\end{proof}

\section{Sequential bubbling} 

\subsection{Sequential bubbling for finite time blow-up solutions}

\begin{prop}[Sequential bubbling for solutions that blow up in finite time]  \label{prop:seq-ftbu} 
Let $\ell, m \in\Z$, $u_0 \in \E_{\ell, m}$, and let $u(t)$ denote the solution to~\eqref{eq:hf} with initial data $u_0$. Suppose that $T_+(u_0) < \infty$. There exist integers $m_{\infty}, m_{\Delta}$, a mapping $u^* \in \E_{0, m_\infty}$,  an integer $N \ge 1$, a sequence of times $t_n \to T_+$, signs $\vec\iota \in \{-1, 1\}^N$, a sequence of scales $\vec \lam_n \in (0, \infty)^N$, and an error $g_n$ defined by 
\EQ{
u(t_n) = m_{\Delta} \pi + \sum_{j =1}^N \iota_j( Q_{\lam_n} - \pi) + u^* + g_n,
}
with the following properties: 
\begin{itemize} 
\item[(i)] The integer $N \ge 1$ and the body map $u^*$ satisfy, 
\EQ{ \label{eq:energy-limit} 
\lim_{t \to T_+} E( u(t)) = N E(Q)  + E( u^*); 
}
\item[(i)] for any $\al>0$, 
\EQ{ \label{eq:N-bubbles-bu} 
\lim_{t \to T_+} E( u(t); 0, \al(T_+ - t)^{\frac{1}{2}}) = NE(Q), 
}
\EQ{ \label{eq:en-ext-bu} 
\lim_{t \to T_+} E( u(t) - u^*; \al(T_+ - t)^{\frac{1}{2}}, \infty) = 0, 
}
and there exists $0< T_0 < T_+$ and function $\rho : [T_0, T_+) \to (0,\infty)$ satisfying, 
\begin{equation} \label{eq:radiation} 
\lim_{t \to T_*} \big((\rho(t) / \sqrt{T_+-t}) + \| u(t) -  u^* - m_\Delta \pi\|_{\cE(\rho(t))}\big) = 0; 
\end{equation}

\item[(ii)] 
the error $g_n$ and the scales $\vec \lam_n$ satisfy, 
\EQ{ \label{eq:d(t_n)} 
\lim_{ n \to \infty} \Big( \| g_n \|_{\E}^2 + \sum_{j =1}^N \Big( \frac{\lam_{n, j}}{\lam_{n, j+1}} \Big)^k  \Big)^{\frac{1}{2}}  = 0 , 
}
where here we adopt the convention that $\lam_{n, N+1} := (T_+ - t_n)^{\frac{1}{2}}$. 
\end{itemize} 
\end{prop}



\begin{lem} [Identification of the body map] \label{lem:body-id} 
Let $u_0 \in \E_{\ell, m}$ and let $u(t)$ be the solution to~\eqref{eq:hf}. Suppose that $T_+(u_0)< \infty$ and let $I_* = [0, T_+)$. There exist $m_{\infty}, m_{\Delta} \in \Z$ and a mapping $u^* \in \E_{0, m_{\infty}}$ such that for any $r_0>0$, 
\EQ{ \label{eq:local-ext-lim} 
\lim_{t \to T_*}\| u(t) -  u^* - m_\Delta \pi\|_{\cE(r \ge r_0)} = 0.
} 
Moreover, there exists $L>0$ such that for each $r_0  \in (0, \infty]$, 
\EQ{ \label{eq:L-def} 
\lim_{t \to T_+} E( u(t); 0, r_0) = L + E(u^*; 0, r_0), 
}
and in particular, $\lim_{r_0 \to 0} \lim_{t \to T_+} E( u(t); 0, r_0) = L$. 

\end{lem} 

\begin{proof}[Proof of Lemma~\ref{lem:body-id}] 
In the general (non-equivariant) setting Struwe~\cite{Struwe85} proves the existence of the body map as the weak limit of the flow in $H^1$ as $t \to T_+$ and moreover that one has strong $C^2$ convergence on compact sets not containing the bubbling points (the origin in our case); see for example~\cite[Step 3, Proof of Theorem 6.16]{Lin-Wang}. The existence of the limit $L$ is proved by Qing in~\cite[Proposition 2.1]{Qing}, and an identical argument can be used in the equivariant setting. 
%
\end{proof} 

\begin{proof}[Proof of Proposition~\ref{prop:seq-ftbu}]
We follow, roughly,  the arguments by Qing in~\cite[Proof of Theorem 1.1]{Qing} and Topping in~\cite[Proof of Theorem 1.4]{Topping-winding}. The main ingredient is the compactness result, Lemma~\ref{lem:compact}. 
Let $u(t) \in \E_{\ell, m}$ be a heat flow blowing up at time $T_+ >0$. By~\eqref{eq:tension-L2} we can find a sequence $t_n \to T_+$ so that, 
\EQ{
(T_+ - t_n)^{\frac{1}{2}} \| \calT( u(t_n)) \|_{L^2}  \to 0 \mas n \to \infty. 
}
 We can now apply Lemma~\ref{lem:compact} with $\rho_n:= (T_+ - t_n)^{\frac{1}{2}}$, which yields $N \ge 0$, $m_0\in \Z$, $\vec \iota \in \{-1, 1\}^N, \vec \lam_n \in (0, \infty)^N$ such that after passing to a subsequence, we have 
\EQ{ \label{eq:decomp-0} 
\lim_{n \to \infty} \Big( \| u(t_n) - \calQ(m_0, \vec \iota, \vec \lam_n) \|_{\E( r \le A (T_+- t_n)^{\frac{1}{2}})}^2 + \sum_{j =1}^{N-1} \Big( \frac{ \lam_{n, j}}{\lam_{n, j+1}} \Big)^{k} \Big) = 0
}
for each $A>0$, and moreover that $\lam_{n, N} \lesssim (T_+- t_n)^{\frac{1}{2}}$. Next, for each $R>0$ define the localized energy, 
\EQ{
\Theta_R(t) := \int_0^\infty \chi_R(r)^2 \bfe(u(t, r)) \, r \, \ud r. 
}
along with the localized energy of the body map, 
\EQ{
\Theta_R^*:= \int_0^\infty \chi_R(r)^2 \bfe(u^*(r)) \, r \, \ud r. 
}
From~\eqref{eq:local-energy} we see that for each $0<s<\tau< T_+$ we have, 
\EQ{ \label{eq:Theta-ineq} 
\Big| \Theta_R(\tau) - \Theta_R(s) \Big| &\lesssim \int_s^{\tau} \| \p_t u(t) \|_{L^2}^2 \, \ud t  + \frac{(\tau - s)^{\frac{1}{2}}}{R} \Big( \int_s^\tau \| \p_t u(t) \|_{L^2}^2 \, \ud  t \Big)^{\frac{1}{2}}\\
&\lesssim  \int_s^{T_+} \| \p_t u(t) \|_{L^2}^2 \, \ud t  + \frac{(T_+ - s)^{\frac{1}{2}}}{R} \Big( \int_s^{T_+} \| \p_t u(t) \|_{L^2}^2 \, \ud  t \Big)^{\frac{1}{2}}
}
Since the right-hand side tends to zero as $s \to T_+$, it follows that $\lim_{t \to T_+}\Theta_R(t) := \ell_R$ exists. Define, 
\EQ{
\frac{1}{2\pi} L_R :=  \ell_R -  \Theta_R^*
}
and we claim that in fact, $L_R = L := \lim_{r_0 \to 0} \lim_{t \to T_+}E( u(t); 0, r_0)$, which is independent of $R>0$. To see this we write, for any $0<r_0 <R$, 
\EQ{
\Theta_R(t) - \Theta_R^*  &= \int_{r_0}^{4R}  \chi_R(r)^2 (\bfe(u(t, r)) - \bfe(u^*(r))) \, r \, \ud r 
 +\frac{1}{2\pi}E( u(t); 0, r_0)    - \frac{1}{2\pi}E(u^*; 0, r_0)
}
Letting $t \to T_+$, the right hand side tends to $\frac{1}{2\pi}L_R$. By~\eqref{eq:local-ext-lim} the first term on the left vanishes as $t \to T_+$. Sending $r_0 \to 0$ after letting $t \to T_+$ on the right, we see from~\eqref{eq:L-def} that $L_R = L = \lim_{r_0 \to 0} \lim_{t \to T_+}E( u(t); 0, r_0)$. 

Next, let $\gamma>0$ and set $R = \gamma(T_+- s)^{\frac{1}{2}}$ in~\eqref{eq:Theta-ineq} we obtain, after letting $\tau \to T_+$, 
\EQ{
\Big| \frac{1}{2 \pi} L + \Theta_{\gamma(T_+ -s)^{\frac{1}{2}}}^* - \Theta_{\gamma(T_+ -s)^{\frac{1}{2}}}(s) \Big| &\lesssim \int_s^{T_+} \| \p_t u(t) \|_{L^2}^2 \, \ud t  + \frac{1}{\gamma} \Big( \int_s^{T_+} \| \p_t u(t) \|_{L^2}^2 \, \ud  t \Big)^{\frac{1}{2}}
}
Letting $s \to T_+$ above we see that $\lim_{s \to T_+}  \Theta_{\gamma(T_+ -s)^{\frac{1}{2}}}(s) = \frac{1}{2\pi} L$ for all $\gamma>0$. 

Let $\al >0$ and note the inequality, 
\EQ{
2 \pi \Theta_{\frac{\al}{2}(T_+ -s)^{\frac{1}{2}}}(s) \le E( u(s); 0, \al (T_+ -s)^{\frac{1}{2}}) \le  2 \pi \Theta_{\al(T_+ -s)^{\frac{1}{2}}}(s)
}
which implies that  $\lim_{s \to T_+}E( u(s); 0, \al (T_+ -s)^{\frac{1}{2}}) = L$ for any $\alpha>0$. Hence, for any $0< \al< A < \infty$, $\lim_{s \to T_+} E(u(s); \al (T_+ -s)^{\frac{1}{2}}, A(T_+ -s)^{\frac{1}{2}}) = 0$.   Returning to the decomposition~\eqref{eq:decomp-0} we find that 
\EQ{ \label{eq:lambda_N/ss} 
\frac{\lam_{n, N} }{(T_+ - t_n)^{\frac{1}{2}}} \to 0 \mas n  \to \infty, 
}
and as a consequence, $L = NE(Q)$ and~\eqref{eq:N-bubbles-bu} is proved. Further, we see from~\eqref{eq:L-def} that for every $r_0>0$, 
\EQ{
\lim_{t \to T_+} E( u(t); 0, r_0) =  N E(Q) + E(u^*; 0, r_0). 
}
and we see from~\eqref{eq:bu-crit} that $N \ge 1$. 
Combining the above with~\eqref{eq:N-bubbles-bu} we see that for every $\alpha>0$, $r_0 \in (0, \infty]$, 
\EQ{ \label{eq:ss-to-r_0} 
\lim_{t \to T_+} E( u(t); \alpha(T_+ - t)^{\frac{1}{2}}, r_0) =  E(u^*; 0, r_0)
}
and~\eqref{eq:energy-limit} now follows. 
Next, if~\eqref{eq:en-ext-bu} were to fail, we could find $\al_1, \eps_1>0$ and a sequence $s_n \to T_+$ such that 
\EQ{
E( u(s_n) - u^*; \al_1 (T_+- s_n)^{\frac{1}{2}}, \infty) \ge \eps_1 , \quad \forall \, n. 
}
To reach a contradiction, we choose $r_0>0$ sufficiently small so that $E( u^*; 0, r_0) \le \eps_1/8$, and then, using~\eqref{eq:local-ext-lim} and~\eqref{eq:ss-to-r_0}, $n$ sufficiently large so that $E(u(s_n) - u^*; r_0, \infty) \le \eps_1/8$ and $E(u(s_n); \al_1 (T_+- s_n)^{\frac{1}{2}}, r_0) \le \eps_1/4$. We then estimate, 
\EQ{
E( u(s_n) - u^*;& \al_1 (T_+- s_n)^{\frac{1}{2}}, \infty) \le E( u(s_n) - u^*; \al_1 (T_+- s_n)^{\frac{1}{2}}, r_0) + E( u(s_n) - u^*; r_0, \infty)\\
&\le 2E( u(s_n); \al_1 (T_+- s_n)^{\frac{1}{2}}, r_0) + 2 E(u^*; \al_1 (T_+- s_n)^{\frac{1}{2}}, r_0) + \eps_1/8  \le 7\eps_1/8, 
}
a contradiction, proving~\eqref{eq:en-ext-bu}. We see from~\eqref{eq:local-ext-lim} and~\eqref{eq:decomp-0} that $m_0 = m_{\De}$ and from Lemma~\ref{lem:pi} we have, 
\EQ{
\lim_{t \to T_+} \| u(t) - u^* - m_{\De} \pi \|_{\E( r \ge \al(T_+ - t) )}  = 0,
}
which implies~\eqref{eq:radiation}. Finally, the above together with~\eqref{eq:decomp-0} and~\eqref{eq:lambda_N/ss} yield~\eqref{eq:d(t_n)}. 
\end{proof} 
 
 
\subsection{Sequential bubbling for global solutions} 

\begin{prop}[Sequential bubbling for global-in-time solutions]  \label{prop:seq-global}  Let $\ell, m \in \Z$. 
Let $u_0 \in \E_{\ell, m}$ and let $u(t)$ denote the solution to~\eqref{eq:hf} with initial data $u_0$. Suppose that $T_+(u_0) = \infty$. 
Then there exist $T_0>0$, an integer $N \ge 0$, a sequence of times $t_n \to \infty$, signs $\vec\iota \in \{-1, 1\}^N$, a sequence of scales $\vec \lam_n \in (0, \infty)^N$, and an error $g_n$ defined by 
\EQ{
u(t_n) = m \pi + \sum_{j =1}^N \iota_j( Q_{\lam_n} - \pi)  + g_n
}
with the following properties:
\begin{itemize} 
\item[(i)] the integer $N \ge 0 $ satisfies,  
\EQ{ \label{eq:energy-limit-global} 
\lim_{t \to \infty} E( u(t)) =  NE(Q); 
}
\item[(ii)] for every $\alpha>0$, 
\EQ{ \label{eq:ext-energy-global} 
\lim_{t \to \infty} E( u(t); \al \sqrt{t}, \infty) = 0, 
} 
and there exists $T_0>0$ and a function $ \rho: [T_0, \infty) \to (0, \infty)$ such that 
\EQ{ \label{eq:ext-E-global} 
\lim_{t \to \infty} \Big( \frac{\rho(t)}{ \sqrt{t}} + \| u(t) - m\pi \|_{\E( r \ge \rho(t))} \Big) = 0; 
}
\item[(iii)] the scales $\vec \lam_n$ and the sequence $g_n$ satisfy, 
\EQ{ \label{eq:global-seq} 
\lim_{ n \to \infty} \Big( \| g_n \|_{\E}^2 + \sum_{j =1}^N \Big( \frac{\lam_{n, j}}{\lam_{n, j+1}} \Big)^k  \Big)^{\frac{1}{2}}  = 0 
}
where here we adopt the convention that $\lam_{n, j+1} :=  t_n^{\frac{1}{2}}$. 
\end{itemize} 
\end{prop}

%

\begin{proof} 
Let $u(t) \in \E_{\ell, m}$ be a heat flow defined globally in time. By~\eqref{eq:tension-L2} we can find a sequence $t_n \to \infty$ so that, 
\EQ{
 t_n^{\frac{1}{2}} \| \calT( u(t_n)) \|_{L^2}  \to 0 \mas n \to \infty. 
}
 We can now apply Lemma~\ref{lem:compact} with $\rho_n:=  t_n^{\frac{1}{2}}$, which yields $N \ge 0$, $m_0\in \Z$, $\vec \iota \in \{-1, 1\}^N, \vec \lam_n \in (0, \infty)^N$ such that after passing to a subsequence, we have 
\EQ{ \label{eq:decomp-global-0} 
\lim_{n \to \infty} \Big( \| u(t_n) - \calQ(m_0, \vec \iota, \vec \lam_n) \|_{\E( r \le A  t_n^{\frac{1}{2}})}^2 + \sum_{j =1}^{N-1} \Big( \frac{ \lam_{n, j}}{\lam_{n, j+1}} \Big)^{k} \Big) = 0
}
for each $A>0$, and moreover that $\lam_{n, N} \lesssim  t_n^{\frac{1}{2}}$.

Fix $\alpha>0$ and let $\eps >0$. By~\eqref{eq:tension-L2} and the fact that $E( u(0))< \infty$  we can find $T_0 = T_0(\eps) >0$ such that, 
\EQ{ \label{eq:T_0-choice} 
 \frac{4\sqrt{E(u(0))}}{\al}\Big(\int_{T_0}^\infty \int_0^\infty (\p_t u (t, r))^2 \, r \, \ud r \, \ud t\Big)^{\frac{1}{2}}  \le   \eps
}
Next, choose $T_1 \ge T_0$ so that 
\EQ{ \label{eq:T_1-choice} 
E( u(T_0); \al \sqrt{T}/4, \infty) \le \eps
}
for all $T \ge T_1$. Fixing any such $T$, we set 
\EQ{
\phi(t, r) = \phi_T(r)  = 1 - \chi( 4r/ \al \sqrt{T}) \mfor t \in [T_0, T]
}
where $\chi(r)$ is a smooth function on $(0, \infty)$ such that $\chi(r) = 1$ for $r \le 1$, $\chi(r) = 0$ if $r \ge 4$, and  $\abs{\chi'(r)} \le 1$ for all $r \in (0, \infty)$. Since $\frac{\ud }{\ud t} \phi(t, r) = 0$ for $t \in [T_0, T]$ it follows from~\eqref{eq:loc-en-ineq} that, 
\EQ{
\int_0^\infty \bfe(u(T, r)) \, \phi_T(r)^2 \, r \, \ud r  \le \int_0^\infty \bfe(u(T_0, r)) \, \phi_T(r)^2 \, r \, \ud r + \frac{4\sqrt{E(u(0))}}{\al}\Big(\int_{T_0}^T \int_0^\infty (\p_t u (t, r))^2 \, r \, \ud r \, \ud t\Big)^{\frac{1}{2}}
} 
Using the above together with~\eqref{eq:T_0-choice} and~\eqref{eq:T_1-choice}  we find that
\EQ{
E( u(T); \al \sqrt{T}, \infty) \le \eps.  
}
for all $T \ge T_1$, completing the proof of~\eqref{eq:ext-energy-global}. It follows from~\eqref{eq:ext-energy-global} that there exists $T_0>0$ and a function $\rho: [T_0, \infty) \to (0, \infty)$ with $\rho(t) \ll \sqrt t$ and  $\lim_{t \to \infty} E( u(t); \rho(t), \infty)  = 0$. Thus,~\eqref{eq:ext-E-global} is a consequence of Lemma~\ref{lem:pi}. 

Returning to the sequential decomposition wee see from~\eqref{eq:decomp-global-0}, the fact that $\lam_{n, N} \lesssim t_n^{\frac{1}{2}},$ and from~\eqref{eq:ext-energy-global} that we must have 
\EQ{
\lim_{n \to \infty} \frac{\lam_{n, N}}{t_n^{\frac{1}{2}}}  = 0. 
}
Then,~\eqref{eq:global-seq} follows from the above,~\eqref{eq:ext-E-global} and~\eqref{eq:decomp-global-0}. Moreover we see that $\lim_{n \to \infty} E( u(t_n))  =  N E(Q)$ and the continuous limit~\eqref{eq:energy-limit-global} then follows from the fact that $E( u(t))$ is non-increasing. 
\end{proof}

\section{Decomposition of the solution and collision intervals} \label{sec:decomposition}

For the remainder of the paper we fix a solution
$ u(t) \in \E_{\ell, m}$ of \eqref{eq:hf},
defined on the time interval $I_*=[0, T_*)$
where $T_* := T_+<\infty$ in the finite time blow-up case and $T_*= \infty$ in the global case. 
Let $ u^* \in \E_{0, m_{\infty}}$ be the body map as defined in Proposition~\ref{prop:seq-ftbu} and in the case of a global solution we adopt the convention that $u^* = 0$. 
Note that $m_\infty = 0$ if $T_* = \infty$. We let $m_\Delta$ be as in Proposition~\ref{prop:seq-ftbu} so that $u(t) \sim m_{\Delta} \pi  + u^*$ in the region $r \gtrsim (T_+ - t)^{\frac{1}{2}}$. To unify notation, we adopt the convention that $m_{\Delta} = m$ in the case of a global solution, so that we may again view $u(t) \sim m_{\Delta}  \pi + u^*$ in the region $r \gtrsim \sqrt{t}$. 
By Propositions~\ref{prop:seq-ftbu} and \ref{prop:seq-global} there exists an integer $N \ge0$ and a sequence of times $t_n \to T_*$ so that $u(t_n)- u^*$ approaches an $N$-bubble as $n \to \infty$. 

We define a localized distance to an $N$-bubble. 

\begin{defn}[Proximity to a multi-bubble]
\label{def:proximity}
For all $t \in I$, $\rho \in (0, \infty)$, and $K \in \{0, 1, \ldots, N\}$, we define
the \emph{localized multi-bubble proximity function} as
\begin{equation}
\bfd_K(t; \rho) := \inf_{\vec \iota, \vec\lam}\bigg( \|  u(t) - u^* - \calQ(m_{\Delta}, \vec\iota, \vec\lambda) \|_{\cE(\rho, \infty)}^2 + \sum_{j=K}^{N}\Big(\frac{ \lam_{j}}{\lam_{j+1}}\Big)^{k} \bigg)^{\frac{1}{2}},
\end{equation}
where $\vec\iota := (\iota_{K+1}, \ldots, \iota_N) \in \{-1, 1\}^{N-K}$, $\vec\lambda := (\lambda_{K+1}, \ldots, \lambda_N) \in (0, \infty)^{N-K}$, $\lambda_K := \rho$ and $\lambda_{N+1} := \sqrt{T_+ -t}$ in the finite time blow-up case and $\lambda_{N+1} := \sqrt{t}$ in the case of a global solution.

The \emph{multi-bubble proximity function} is defined by $\bfd(t) := \bfd_0(t; 0)$.
\end{defn}

\begin{rem} 
We emphasize that if $\bfd_K(t; \rho)$ is small, this means that $ u(t) - u^*$ is close to $N-K$ bubbles in the exterior region  $r \in (\rho, \infty)$. 
\end{rem} 
We can now rephrase a consequence of Propositions~\ref{prop:seq-ftbu} and~\ref{prop:seq-global} in this notation: there exists a monotone sequence $t_n \to T_*$ such that
\begin{equation}
\label{eq:dtn-conv}
\lim_{n \to \infty} \bfd(t_n) = 0.
\end{equation}

We state and prove some simple consequences of the set-up above.
We always assume $N \geq 1$, since Theorem~\ref{thm:main} in the case $N = 0$ is immediate from~\eqref{eq:energy-limit-global}. 

A direct consequence of~\eqref{eq:ext-E-global} is that $ u(t)$ always approaches a $0$-bubble in some exterior region. With $\rho_N(t) = \rho(t)$ given by the function in Proposition~\ref{prop:seq-ftbu} or~\ref{prop:seq-global} the following lemma is  immediate from the conventions of Definition~\ref{def:proximity}. 
\begin{lem}
\label{lem:conv-rhoN}
There exists $T_0>0$ and  function $\rho_N: [T_0, T_*) \to (0, \infty)$ such that
\begin{equation}
\label{eq:conv-rhoN}
\lim_{t\to T_*}\bfd_N(t; \rho_N(t)) = 0.
\end{equation}
\end{lem}

\subsection{Collision intervals} \label{ssec:collision} 

Theorem~\ref{thm:main} will follow from showing that, 
\begin{equation}
\label{eq:dt-conv}
\lim_{t \to T_*} \bfd(t) = 0.
\end{equation}
The approach which we adopt in order to prove~\eqref{eq:dt-conv} is to study colliding bubbles.
A collision is defined as follows.
\begin{defn}[Collision interval]
\label{def:collision}
Let $K \in \{0, 1, \ldots, N\}$. A compact time interval $[a, b] \subset I_*$ is a \emph{collision interval}
with parameters $0 < \epsilon < \eta$ and $N - K$ exterior bubbles if
\begin{itemize}
\item $\bfd(a) \leq \epsilon$ and $\bfd(b) \ge \eta$,
\item there exists a function $\rho_K: [a, b] \to (0, \infty)$ such that $\bfd_K(t; \rho_K(t)) \leq \epsilon$
for all $t \in [a, b]$.
\end{itemize}
In this case, we write $[a, b] \in \calC_K(\epsilon, \eta)$.
\end{defn}
\begin{defn}[Choice of $K$]
\label{def:K-choice}
We define $K$ as the \emph{smallest} nonnegative integer having the following property.
There exist $\eta > 0$, a decreasing sequence $\epsilon_n \to 0$,
and sequences $(a_n), (b_n)$ such that $[a_n, b_n] \in \calC_K(\epsilon_n, \eta)$ for all $n \in \{1, 2, \ldots\}$.
\end{defn}
\begin{lem}[Existence of $K \ge 1$]
\label{lem:K-exist}
If \eqref{eq:dt-conv} is false, then $K$ is well defined and $K \in \{1, \ldots, N\}$.
\end{lem}

\begin{rem} 
The fact that $K \ge 1$ means that at least one bubble must lose its shape if~\eqref{eq:dt-conv} is false.
\end{rem} 

\begin{proof}[Proof of Lemma~\ref{lem:K-exist}]
Assume \eqref{eq:dt-conv} does not hold, so that there exist $\eta > 0$ and a monotone sequence $b_n \to T_*$ such that
\begin{equation}
\bfd(b_n) \geq \eta, \qquad\text{for all }n.
\end{equation}
We claim that there exist sequences $(\epsilon_n), (a_n)$ such that $[a_n, b_n] \in \calC_N(\epsilon_n, \eta)$.
Indeed, \eqref{eq:dtn-conv} implies that there exist $\epsilon_n \to 0$ and $a_n \leq b_n$ 
such that $\bfd(a_n) \leq \epsilon_n$. Note that $a_n \to T_*$ and $b_n \to T_*$.
Let $\rho_N: [a_n, b_n] \to (0, \infty)$ be the function given by Lemma~\ref{lem:conv-rhoN},
restricted to the time interval $[a_n, b_n]$.
Then \eqref{eq:conv-rhoN} yields
\begin{equation}
\lim_{n\to\infty}\sup_{t\in[a_n, b_n]}\bfd_N(t; \rho_N(t)) = 0.
\end{equation}
Upon adjusting the sequence $\epsilon_n$, we obtain that all the requirements of Definition~\ref{def:collision}
are satisfied for $K = N$.

We now prove that $K \geq 1$. Suppose $K = 0$. By  Definition~\ref{def:collision} of a collision interval, there exist $\eta > 0$, and sequences $a_n, b_n \to T_*$ and $\rho_0(b_n) \geq 0$ such that
$\bfd_0(b_n;  \rho_0(b_n)) \leq \epsilon_n$ and at the same time $\bfd(b_n) \geq \eta$. We show that this is impossible. 


Define $ v_n := u(b_n) - u^*$. 
Since $\bfd_0(b_n;  \rho_0(b_n)) \leq \epsilon_n$ we can find parameters, $\rho_0(b_n) \ll  \lam_{n, 1} \ll \dots \ll \lam_{n, N}$ and signs $\vec{\iota}_n$ such that defining $ { g}_n = v_n -  \calQ( m_\De, \vec{ \iota}_n, \vec{ \lam}_n)$ we have
\EQ{ \label{eq:tig-small} 
\bfd_0(c_n; \rho_0(b_n)) \simeq \| {g}_n \|_{\E( \rho_0(b_n), \infty)}^2 + \sum_{j =0}^N\Big( \frac{  \lam_{n, j}}{\lam_{n, j+1}}\Big)^k \lesssim \eps_n^2. 
}
If $T_*< \infty$, with $\rho(t)$ as in~\eqref{eq:radiation} we see that we must have $\lam_{n, N} \ll  \rho(b_n) \ll (T_*-b_n)^{\frac{1}{2}}$, and thus using~\eqref{eq:radiation} along with~\eqref{eq:tig-small} and Lemma~\ref{lem:M-bub-energy} we have
\EQ{
E(  u(b_n); \rho_0(b_n), \infty) &= E(  {g}_n + u^* +  \calQ( m_\De, \vec{ \iota}_n, \vec{ \lam}_n); \rho_0(b_n), \rho(b_n))  \\
&\quad + E(  { g}_n +  u^* +  \calQ( m_\De, \vec{\iota}_n, \vec{\lam}_n); \rho(b_n), \infty) \\
& = N E(  Q) + E( u^*)  + o_n(1) . 
}
A similar argument in the case $T_*= \infty$ shows that 
\EQ{
E(  u(b_n); \rho_0(b_n), \infty) & = N E(  Q)  + o_n(1) . 
}
Since by~\eqref{eq:energy-limit} and~\eqref{eq:energy-limit-global} we know that $\lim_{ n \to \infty} E( u(b_n)) = N E(  Q) + E(  u^*)$, we conclude from the previous line that, 
\EQ{
E(  u(b_n); 0, \rho_0(b_n)) = o_n(1) \mas n \to \infty. 
}
Using the fact that $\rho_0(b_n) \ll \rho(b_n)$ it follows that $E(  v_n; 0, \rho_0(b_n)) = o_n(1)$, and hence by~\eqref{eq:H-E-comp} we conclude that 
\EQ{
\| v_n - \ell  \pi \|_{\E( 0, \rho_0(b_n))} \lesssim E(  v_n; 0, \rho_0(b_n)) = o_n(1) \mas n \to \infty
}
Thus, combining the above with~\eqref{eq:tig-small} we have $\bfd(b_n) = o_n(1)$ as $n \to \infty$, a contradiction. 
\end{proof} 

\begin{rem} \label{rem:collision} 
For each collision interval  we may assume without loss of generality that $\bfd(a_n)  = \eps_n$,  $\bfd(b_n) = \eta$, and $\bfd(t) \in [\eps_n, \eta]$ for each $t \in [a_n, b_n]$. Indeed, given some initial choice of $[a_n, b_n] \in \calC_K( \eps_n, \eta)$, 
just set $a_n \le \ti a_n := \sup\{ t \in [a_n,  b_n] \mid \bfd(t) \le \eps_n \}$ and $\ti b_n := \inf\{t \in [\ti a_n, b_n] \mid \bfd(t) \ge \eta\}$. 

Similarly, given some initial choice $\eps_n \to 0, \eta>0$ and intervals $[a_n, b_n] \in \calC_K( \eta, \eps_n)$ we are free to ``enlarge'' $\eps_n$ or ``shrink'' $\eta>0$, by choosing some other sequence $\eps_n \le \ti \eps_n  \to 0$, and $0< \ti \eta \le \eta$, and new collision subintervals $[\ti a_n, \ti b_n]  \subset [a_n, b_n] \cap \calC_{K}(\ti \eta, \ti \eps_n)$ as in the previous paragraph. We will enlarge our initial choice of $\eps_n$ and shrink $\eta$ in this fashion  over the course of the proof. 
\end{rem} 

%

\subsection{Decomposition of the solution} 

\begin{lem}[Basic modulation] \label{lem:modulation} 
Let $K \ge 1$ be the number given by Lemma~\ref{lem:K-exist}. There exist $\eta>0$, a sequence $\eps_n \to 0$,  and sequences $a_n, b_n \to \infty$ satisfying the requirements of Definition~\ref{def:K-choice}, and such that $\bfd(a_n)  = \eps_n$, $\bfd(b_n) = \eta$ and $\bfd(t) \in [\eps_n, \eta]$ for all $t \in [a_n, b_n]$ and so that the following properties hold. There exist signs $\vec \iota \in \{-1, 1\}^N$,  a function $\vec \lambda = ( \lambda_1, \dots, \lam_N) \in C^1(\cup_{n \in \N} [a_n, b_n];  (0, \infty)^{N})$, sequences $\al_n \to 0$ and $\nu_n \to 0$,  such that defining the functions, 
\EQ{ \label{eq:nu-def} 
\nu:\cup_{n \in \N} [a_n, b_n] \to (0, \infty), \quad  \nu(t):= \nu_n \lam_{K+1}(t), \mfor \, \, t\in[a_n,b_n],
} 
\EQ{
\al \cup_{n \in \N} [a_n, b_n] \to (0, \infty), \quad  \al(t):= \begin{cases}  \al_n\sqrt{T_+ - t_n}\mif T_+<\infty \\ \al_n \sqrt{t} \mif T_+ = \infty \end{cases}, \mfor \, \, t\in[a_n,b_n], 
}
\EQ{ \label{eq:u^*(t)-def} 
u^*(t) :=  \begin{cases} (1 - \chi_{\al(t)}) \big( u(t) - m_{\Delta} \pi \big)\mif T_+<\infty \\ 0 \mif T_+= \infty \end{cases}  
}
and 
\EQ{
g: \cup_{n \in \N} [a_n, b_n] \to  \E; \quad g(t) := u(t) - u^*(t) - \calQ(m_\Delta, \vec \iota, \vec \lambda(t)), 
}
there hold,
\begin{itemize} 
\item[(i)] the orthogonality conditions, 
\EQ{ \label{eq:g-ortho} 
0 = \La \calZ_{\U{\lambda(t)}} \mid g(t) \Ra, \quad \forall \, t \in [a_n, b_n], \quad \forall n; 
}
\item[(ii)] and the estimates, 
\EQ{ \label{eq:nu-estimates} 
\lim_{n \to \infty} \sup_{t \in [a_n, b_n]} \Big(  \frac{\nu(t)}{ \lam_{K+1}(t)}  + \sum_{j=K+1}^{N-1}  \frac{\lam_{j}(t)}{\lam_{j+1}(t)}  + \frac{\lam_{N}(t)}{\al(t)}+ E( u(t); \frac{1}{4}{\nu(t)},  4 \nu(t))  \Big) = 0,
}
\EQ{  \label{eq:d-g-lam} 
C_0^{-1}\bfd(t)  \leq \| g(t) \|_{\cE} +  \sum_{j=1}^{N-1} \Big( \frac{ \lam_{j}(t)}{\lam_{j+1}(t)} \Big)^{\frac{k}{2}} 
\leq C_0\bfd(t), 
}
\EQ{ \label{eq:g-refined}
\| g(t) \|_{\E} + \sum_{j \not \in \calA} \Big( \frac{\lam_j}{\lam_{j+1}} \Big)^{\frac{k}{2}} \le C_0  \sum_{j  \in \calA} \Big( \frac{\lam_j}{\lam_{j+1}} \Big)^{\frac{k}{2}}
}
\EQ{ \label{eq:lambda'-bound} 
\abs{\lam_j'(t)} \le C_0 \frac{1}{\lam_j(t)}  \bfd(t), 
}
for all $t \in [a_n, b_n]$ and all $n \in \N$; 
\item[(iii)] for any sequence $s_n \in [a_n, b_n]$ and any sequence $R_n$ such that $\nu(s_n) \le R_n \ll \lam_{K+1}(s_n)$  if $K < N$ and $\nu(s_n) \le R_n \le \al(s_n)$  if $K=N$, then, 
\EQ{ \label{eq:N-K-bubbles} 
\lim_{n \to \infty} E( u(s_n); R_n, \infty) = (N-K) E(Q) + E( u^*). 
}
and, 
\EQ{ \label{eq:N-K-converge} 
\lim_{n \to \infty}  \Big( \| u(s_n) - u^*(s_n) - \calQ( m_{\De}, \iota_{K+1}, \dots, \iota_N, &\lam_{K+1}(s_n), \dots, \lam_{N}(s_n)) \|_{\E( r \ge R_n)} \\
& + \sum_{j=K+1}^N \Big( \frac{\lam_{j}(s_n)}{\lam_{j+1}(s_n)} \Big)^{\frac{k}{2}} \Big) = 0. 
}
\end{itemize} 
\end{lem}

%
%
%
%

\begin{rem} 
One should think of $\nu(t)$ as the scale that separates the $N-K$ ``exterior'' bubbles, which stay coherent on the union of the collision intervals $[a_n, b_n]$  from the $K$ ``interior'' bubbles that are coherent at the left endpoint $[a_n, b_n]$, but come into collision inside the interval and lose their shape. In the case $K =N$, there are no exterior bubbles, we set $\lam_{K+1}(t) := \sqrt{T_+ - t}$ and $\nu_n \to 0$ is chosen using~\eqref{eq:radiation} in the blow up case, and  $\lam_{K+1}(t) := \sqrt{t}$ and $\nu_n \to 0$ is chosen using~\eqref{eq:ext-E-global} in the global case. 
\end{rem}

\begin{proof}[Proof of Lemma~\ref{lem:modulation}] We carry out the argument in the case $T_+< \infty$,  and note that the global case is similar, and in fact, slightly less involved since $u^*= 0$ in that case. 
Let $a_n,  b_n, \eps_n,  \eta$, and $K\in \{1, \dots, N\}$  be some initial choice of parameters given by Definition~\ref{def:K-choice} and Lemma~\ref{lem:K-exist}. Over the course of the proof we will shrink $\eta$ and enlarge $\eps_n$ as in Remark~\ref{rem:collision}, but abuse notation by still denoting the resulting subintervals by $[a_n, b_n]$ after these modifications.

We first define the function $\al(t)$ and choose the sequence $\nu_n \to 0$. By Defintion~\ref{def:proximity}, for each $n$ we can find scales $\rho_K(t) \ll  \mu_{K+1}(t)  \ll \dots \ll   \mu_{N}(t) \ll (T_+-t)^{\frac{1}{2}} $ and signs $\vec \s(t) \in \{-1, 1\}^{N-K}$  for $t \in [ a_n,   b_n]$, such that defining $ h_{\rho_K}(t)$ for $r \in ( \rho_K(t), \infty)$ by 
\EQ{
u(t) - u^* = \calQ (m_{\De}, \vec \s(t), \vec{ \mu}(t)) + h_{\rho_K}(t) 
}
we have, 
\EQ{ \label{eq:timu-small} 
\bfd(t; \rho_K(t)) \simeq \| h_{\rho_K}(t) \|_{\E( \rho_K(t), \infty)}^2 + \sum_{j =K}^N \Big( \frac{ \mu_{j}(t)}{  \mu_{j+1}(t)} \Big)^k  \lesssim \eps_n^2 , 
}
keeping the convention $\mu_K(t) := \rho_K(t),  \mu_{N+1}(t) :=(T_+-t)^{\frac{1}{2}}$. 
Using $\lim_{n \to \infty} \sup_{t \in [a_n, b_n]}\bfd_K( t; \rho_K(t)) = 0$ and the fact that 
\EQ{ \label{eq:ext-bubble-en} 
\lim_{n \to \infty} \sup_{t \in [a_n,  b_n]}E( \calQ (m_{\De}, \vec \sigma(t), \vec{ \mu}(t)) ; \nu_{n, 1}  \ti \mu_{K+1}(t), \nu_{n, 2} \ti \mu_{K+1}(t)) = 0, 
}
for any two sequence $\nu_{n, 1}  \ll \nu_{n, 2} \ll 1$, 
we can choose a sequence $\nu_n \to 0$ such that for any $A>1$, 
\EQ{ \label{eq:numu} 
\rho_{K}(t) \le \nu_n \mu_{K+1}(t), \mand  \lim_{n \to \infty}\sup_{t \in [a_n, b_n]} E( u(t) -  u^*; \frac{1}{A}\nu_n  \mu_{K+1}(t), A \nu_n    \mu_{K+1}(t)) =  0. 
}
Next, letting $\rho(t)$ be as in~\eqref{eq:radiation}, we can use \eqref{eq:timu-small} to choose $\al_n \to 0$ to be a sequence such that, 
\EQ{ \label{eq:alpha-def} 
\lim_{n \to \infty} \sup_{t \in [a_n, b_n]} \Big( \frac{\mu_{N}(t)}{ \al_n (T_+ - t)^{\frac{1}{2}}} +  \frac{\rho(t)}{\al_n (T_+ - t)^{\frac{1}{2}}} \Big)  = 0,
}
and we define $\al(t) := \al_n (T_+ - t)^{\frac{1}{2}}$ for $t \in [a_n, b_n]$. If $K=N$ we may assume that $\al_n \ge \nu_n$. Setting, 
\EQ{ \label{eq:u^*(t)-def-1} 
u^*(t):= (1 - \chi_{\al(t)}) \big( u(t) - m_{\Delta} \pi \big)
}
we see from~\eqref{eq:radiation} and the fact that $\lim_{t \to T_+} E( u^*; \gamma(t))  = 0$ for any $\gamma(t) \to 0$ as $t \to T_+$, that 
\EQ{ \label{eq:u^*(t)-u^*} 
\lim_{t \to T_+}\| u^*(t) - u^*\|_{\E} = 0, 
} 
and by definition, 
\EQ{
u(t) - u^*(t) = \chi_{\al(t)} u(t) + (1- \chi_{\al(t)}) m_{\De} \pi 
}
and by~\eqref{eq:en-ext-bu} and~\eqref{eq:N-bubbles-bu}  we have, 
\EQ{ \label{eq:N-bubbles} 
\lim_{n \to \infty} \sup_{t \in [a_n, b_n]} \big| E( u(t) - u^*(t)) - NE(Q)\big| = 0
}

Now that $u^*(t)$ is defined, we find the parameters $\vec \iota \in \{-1, 1\}^N$ and $\vec \lambda(t) \in (0, \infty)^N$. 
By the definition of $\bfd(t)$ we make an initial choice of signs $\vec{\ti \iota}(t) \in \{-1, 1\}$ and scales $\vec{\ti \lambda}(t) \in (0, \infty)^N$ such that defining 
\EQ{ \label{eq:tig} 
\ti g(t) := u(t) - u^*  - \calQ( m_{\De}, \vec{\ti \iota}(t), \vec{\ti \lambda}(t))
}
we have, 
\EQ{ \label{eq:ti-lambda} 
\bfd(t) \le \| \ti g(t) \|_{\E} + \sum_{j = 1}^{N} \Big( \frac{\ti \lam_{j}(t)}{ \ti \lam_{j+1}(t)} \Big)^{\frac{k}{2}} \le 2 \bfd(t) \le 2 \eta
}
keeping the convention that $\lam_{N+1}(t) = (T_+-t)^{\frac{1}{2}}$. 

By~\eqref{eq:u^*(t)-u^*}~\eqref{eq:tig}, and~\eqref{eq:ti-lambda} we see that $\bfd(t) \le \eta$ implies that 
\EQ{ \label{eq:d-d-comp}
\bfd_{m_{\Delta}, N}(u(t) - u^*(t)) \le C_0 \bfd(t) + o_n(1) \le 2 C_0 \eta,
}
where $\bfd_{m_{\Delta}, N}$ is as in~\eqref{eq:d-def} and $o_n(1)$ denotes a term that tends to zero as $n \to \infty$. 
We may then shrink $\eta>0$ as in Remark~\ref{rem:collision} small enough so that we can apply Lemma~\ref{lem:mod-static} to $u(t) - u^*(t)$, obtaining $\vec \lambda(t) \in (0, \infty)^N$ defined on $\cup_n [a_n, b_n]$, and signs  $\vec \iota \in \{-1, 1\}^N$ (which can be taken independent of $t \in [a_n, b_n]$ using continuity of the flow and independently of $n$ after passing to a subsequence of the $[a_n, b_n]$), and $g(t)$ so that 
\EQ{ \label{eq:g-def} 
u(t) - u^*(t) =  m_\Delta \pi + \sum_{j=1}^N \iota_{ j} (Q_{\lam_{j}(t)} - \pi) + g(t),  \quad \La \calZ_{\U{\lambda(t)}} \mid g(t)\Ra = 0, \quad \forall t \in [a_n, b_n], 
}
and, 
\EQ{
\bfd_{m_{\Delta}, N}(u(t) - u^*(t))) \le \| g(t) \|_{\E} + \sum_{j=1}^{N-1} \Big( \frac{\lam_j(t)}{\lam_{j+1}(t)} \Big)^{\frac{k}{2}} \le C_0 \bfd_{m_{\Delta}, N}(u(t) - u^*(t)) \\
}
Using again~\eqref{eq:u^*(t)-u^*} along with~\eqref{eq:d-d-comp} we see that in fact, 
\EQ{
\bfd(t) - \zeta_{1, n}  \le \| g(t) \|_{\E} + \sum_{j=1}^N \Big( \frac{\lam_j(t)}{\lam_{j+1}(t)} \Big)^{\frac{k}{2}} \le C_0 \bfd(t) +  \zeta_{1, n} 
}
where $\zeta_{1, n}$ is a sequence tending to zero as $n \to \infty$. 
By enlarging $\eps_n$ so that $\eps_n \ge 2 \zeta_{1, n}$ for all $n$ as in Remark~\ref{rem:collision} we prove~\eqref{eq:d-g-lam}.

Next, we compare the scales $\lam_{K+1}, \dots, \lam_{N}$ to $\mu_{K+1}, \dots, \mu_N$. 
Denoting by $\ti \nu(t) := \nu_n \mu_{K+1}(t)$ we claim that for each $j = 1, \dots, N$, 
\EQ{ \label{eq:tinu-lambda} 
\lim_{ n \to \infty} \sup_{t \in [a_n, b_n]} \Big( \frac{\ti \nu(t)}{ \lam_j(t)} + \frac{\lam_j(t)}{\ti \nu(t)} \Big) = 0. 
}
If not, we could find $C>0$, $j \in \{1, \dots, N\}$, a subsequence of the $[a_n, b_n]$ and   a sequence $s_n \in [a_n, b_n]$ such that
\EQ{
C^{-1} \ti \nu(s_n) \le \lam_j(s_n) \le C \ti \nu(s_n)
} 
By~\eqref{eq:d-g-lam} for all $\eta>0$ sufficiently small we can find $\de = \de(\eta), R= R(\eta)>0$ so that  for all $n$, 
\EQ{
\de \le E( u(s_n) - u^*(s_n); R^{-1} \lam_j(s_n), R \lam_j(s_n))  \le E( u(s_n) - u^*(s_n); C^{-1}R^{-1} \ti \nu(s_n), RC \ti \nu(s_n))
}
which contradicts~\eqref{eq:numu}. 

By~\eqref{eq:numu} and Lemma~\ref{lem:pi} we can find integers $m_n$ so that denoting 
\EQ{
w(t) = m_n \pi \chi_{\ti \nu(t)} + ( 1- \chi_{\ti \nu(t)})( u(t) - u^*(t)) 
}
we have, 
\EQ{ \label{eq:exterior-bubbles} 
\| w(t) - \calQ(m_\De, \vec \sigma(t), \vec \mu(t)) \|_{\E}^2 + \sum_{j=K+1}^{N-1} \Big(\frac{ \mu_{j}(t)}{\mu_{j+1}(t)} \Big)^k = o_n(1)
}
On the other hand, by~\eqref{eq:tinu-lambda} we can find $j_0 \in \{1, \dots, N-1\}$ so that 
\EQ{
\| w(t) - \calQ(m_\De, \iota_{j_0}, \dots, \iota_{N}, \lambda_{j_0}(t), \dots, \lambda_{N}(t)) \|_{\E}^2 + \sum_{j=j_0}^{N-1} \Big(\frac{ \lambda_{j}(t)}{\lambda_{j+1}(t)} \Big)^k \le C_0 \eta
}
An application of Lemma~\ref{lem:bub-config} yields $j_0 = K+1$, $\vec\sigma(t) = \{\iota_{K+1}, \dots, \iota_K\}$  and moreover, by shrinking $\eta>0$, we can ensure that 
\EQ{
\sup_{t \in [a_n, b_n]}  \Big| \frac{\lam_j(t)}{\mu_j(t)} - 1 \Big| \le \frac{1}{4} 
}
and thus, defining $\nu(t) := \nu_n \lam_{K+1}(t)$ we see that~\eqref{eq:nu-estimates} follows from~\eqref{eq:timu-small}~\eqref{eq:numu}, and~\eqref{eq:alpha-def}. Let $s_n\in [a_n, b_n]$ and $R_n$ so that $\nu(s_n) \le R_n \ll \lam_{K+1}(s_n)$.   If $K < N$ then $R_n \ll \al(s_n)$, thus, using~\eqref{eq:exterior-bubbles} and~\eqref{eq:alpha-def}, we see that 
\EQ{
E( u(s_n); R_n, \al(s_n))  \to (N-K) E(Q) \mas n \to \infty
}
Since by~\eqref{eq:alpha-def},~\eqref{eq:u^*(t)-def-1} and~\eqref{eq:u^*(t)-u^*}, 
\EQ{
E( u(s_n); \al(s_n), \infty)  \to E( u^*) \mas n \to \infty
}
we see that~\eqref{eq:N-K-bubbles} follows.  If $K =N$ then $E( u(s_n); R_n, \infty) \to E( u^*)$. Similarly $N-K$ converge now follows from~\eqref{eq:exterior-bubbles}. 

Next we prove~\eqref{eq:g-refined}. An application of~\eqref{eq:g-bound-A} together with~\eqref{eq:N-bubbles} gives, 
\EQ{
\| g(t) \|_{\E} + \sum_{j \not \in \calA} \Big( \frac{\lam_j(t)}{\lam_{j+1}(t)} \Big)^{\frac{k}{2}} \le C_0  \sum_{j  \in \calA} \Big( \frac{\lam_j(t)}{\lam_{j+1}(t)} \Big)^{\frac{k}{2}} + \zeta_{2, n} 
}
for some sequence $\zeta_{2, n} \to 0$, which is independent of $t \in [a_n, b_n]$. But then by enlarging $\eps_n \to 0$ as in Remark~\ref{rem:collision} so that $\eps_n \gg \zeta_{2, n}$ we obtain~\eqref{eq:g-refined} via the above and~\eqref{eq:d-g-lam}. 

Lastly, we prove the modulation estimate~\eqref{eq:lambda'-bound}. Differentiating in time the orthogonality conditions~\eqref{eq:g-ortho} yields, for each $j = 1, \dots, N$, the identity, 
\EQ{ \label{eq:diff-ortho} 
\La \p_t g \mid \calZ_{\U{\lam_j}} \Ra = \frac{\lam_j'}{\lam_j} \La \U \calZ_{\U{\lam}} \mid g \Ra
}
Next, differentiating in time the expression for $g(t)$ in~\eqref{eq:g-def}  and recalling the definition of $u^*(t)$ gives, 
\EQ{
\p_t g &= \p_t  \chi_{\al} - \frac{\al'}{\al} \Lam \chi_{\al} ( u(t) - m_{\De} \pi)  + \sum_{j=1}^N \iota_j \lam_j' \Lam Q_{\U{\lam_j}} \\
& =  (\De u) \chi_\alpha - \frac{k^2}{r^2} f(u) \chi_\alpha  - \frac{\al'}{\al} \Lam \chi_{\al} ( u(t) - m_{\De} \pi)  + \sum_{j=1}^N \iota_j \lam_j' \Lam Q_{\U{\lam_j}}\\
& =  \De (\chi_\al u + (1- \chi_\al) m_{\De} \pi) - \frac{k^2}{r^2}f\big(  \chi_\al u + (1- \chi_\al) m_{\De} \pi \big)+ \sum_{j=1}^N \iota_j \lam_j' \Lam Q_{\U{\lam_j}}\\
&\quad - ( u - m_\De \pi) \De \chi_\alpha - 2 \p_r u \p_r \chi_{\alpha} - \frac{\al'}{\al} \Lam \chi_{\al} ( u(t) - m_{\De} \pi)  \\
&\quad - \frac{k^2}{r^2}\Big( f(u) \chi_\alpha - f(  \chi_\al u + (1- \chi_\al) m_{\De} \pi)\Big) , 
}
and we see that 
\EQ{ \label{eq:g-eq} 
\p_t g =  - \LL_{\calQ}  g  + \sum_{j=1}^N \iota_j \lam_j' \Lam Q_{\U{\lam_j}} + f_{\bfi}( m_\De, \vec \iota, \vec \lam) + f_{\bfq}(m_\De, \vec \iota, \vec \lam, g)
 + \phi(u, \al) 
}
where 
\EQ{
\phi(u, \al) &:= - ( u - m_\De \pi) \De \chi_\alpha - 2 \p_r u \p_r \chi_{\alpha} - \frac{\al'}{\al} \Lam \chi_{\al} ( u(t) - m_{\De} \pi)  \\
&\quad - \frac{k^2}{r^2}\Big( f(u) \chi_\alpha - f(  \chi_\al u + (1- \chi_\al) m_{\De} \pi)\Big)
}
and 
\EQ{
f_{\bfi}( m_\De, \vec \iota, \vec \lam) &:=- \uD E( \calQ(m, \vec \iota, \vec \lam))=  - \frac{k^2}{r^2} \Big( f\big( \calQ(m_\De, \vec \iota, \vec \lam)\big) - \sum_{j =1}^{N} \iota_j  f(Q_{\lam_{j}}) \Big) \\
f_{\bfq}(m_\De, \vec \iota, \vec \lam, g) &:= - \frac{k^2}{r^2} \Big( f\big( \calQ(m_\De, \vec \iota, \vec \lam) + g \big)  - f\big(\calQ(m_\De, \vec \iota, \vec \lam) \big) -  f'\big( \calQ(m_\De, \vec \iota, \vec \lam)\big) g \Big). 
}
The subscript  $\bfi$ above stands for ``interaction'' and $\bfq$ stands for ``quadratic.'' 

We make use of the estimates, 
\EQ{ \label{eq:fi-fq} 
\| f_{\bfi}( m_\De, \vec \iota, \vec \lam) \|_{L^1} \lesssim \sum_{j=1}^{N-1}\Big( \frac{\lam_j}{\lam_{j+1}} \Big)^k , \quad \| f_{\bfq}(m_\De, \vec \iota, \vec \lam, g) \|_{L^1} \lesssim  \| g \|_{\E}^2
}
For the $f_{\bfi}$ estimate we expand 
to obtain the expression, 
\EQ{ 
\frac{r^2}{k^2} \uD E ( \calQ(m, \vec \iota, \vec \lam))&= \frac{1}{2} \sin (2\sum_{ i  =2}^M  \iota_i Q_{\lam_i} + 2 \iota_1 Q_{\lam_1})  - \frac{1}{2}\sum_{ i = 1}^M \iota_i \sin 2 Q_{\lam_i} \\
&= - \sin \big(2\sum_{ i =2}^M  \iota_i Q_{\lam_i}\big) \sin^2 Q_{\lam_1} -  \iota_1\sin^2 \big(\sum_{ i =2}^M  \iota_i Q_{\lam_i}\big) \sin 2 Q_{\lam_1}   \\
&\quad + \frac{1}{2} \sin (2\sum_{ i =2}^M  \iota_i Q_{\lam_i})  - \frac{1}{2} \sum_{i =2}^M \iota_i \sin 2 Q_{\lam_i} 
}
Iterating this expansion in the last line above and using the identity $k\sin Q = \Lam Q$ we obtain the pointwise estimates, 
\EQ{ \label{eq:DE-bound} 
|\uD E( \calQ(m, \vec \iota, \vec \lam))| 
&\lesssim \frac{1}{r^2} \sum_{i,j, \ell\, \, \textrm{not all equal}} \Lam Q_{\lam_i} \Lam Q_{\lam_j} \Lam Q_{\lam_\ell} 
}
from which the estimate for $f_{\bfi}$ in~\eqref{eq:fi-fq} follows by way of Lemma~\ref{lem:cross-term}. The estimate for $f_{\bfq}$ in~\eqref{eq:fi-fq} is straightforward.  

For each $j \in \{1, \dots, N\}$ we pair~\eqref{eq:g-eq} with $\calZ_{\U{\lam_j}}$ and use~\eqref{eq:diff-ortho} to obtain the following system 
\begin{multline} 
\iota_j \lam_j' \Big( \La \Lam Q \mid \calZ \Ra - \frac{\iota_j}{\lam_j} \La \calZ_{\U{\lam_j}} \mid g\Ra \Big) + \sum_{i \neq j} \iota_i  \lam_i'\La \Lam Q_{\U{\lam_i}} \mid \calZ_{\U{\lam_j}} \Ra \\
= \La \calL_{\calQ} g \mid \calZ_{\U{\lam_j}} \Ra - \La f_{\bfi}( m_\De, \vec \iota, \vec \lam)\mid \calZ_{\U{\lam_j}} \Ra - \La f_{\bfq}(m_\De, \vec \iota, \vec \lam, g)\mid \calZ_{\U{\lam_j}} \Ra - \La \phi(u, \al) \mid \calZ_{\U{\lam_j}} \Ra. 
\end{multline} 
The above is diagonally dominate for all sufficiently small  $\eta>0$, hence invertible. We note the brutal estimates, 
\EQ{\label{eq:fi}
\Big|  \La \calL_{\calQ} g \mid \calZ_{\U{\lam_j}} \Ra  \Big| &\lesssim \frac{1}{\lam_j}  \|g \|_{\E} \\
\Big|  \La f_{\bfi}( m_\De, \vec \iota, \vec \lam)\mid \calZ_{\U{\lam_j}} \Ra  \Big| &\lesssim  \frac{1}{\lam_j} \sum_{j=1}^{N-1}\Big( \frac{\lam_j}{\lam_{j+1}} \Big)^k   \\
\Big| \La f_{\bfq}(m_\De, \vec \iota, \vec \lam, g)\mid \calZ_{\U{\lam_j}} \Ra  \Big| &\lesssim \frac{1}{\lam_j}  \| g \|_{\E}^2  \\
\Big|\La \phi(u, \al) \mid \calZ_{\U{\lam_j}}  \Ra  \Big|  &= \frac{1}{\lam_j} o_n(1) 
}
We remark that to prove the second inequality in~\eqref{eq:fi} we may use~\eqref{eq:DE-bound} and the definition of $f_{\bfi}$.  
The estimates of the remaining estimates are straightforward and we omit the proofs. 
It follows that, 
\EQ{
\abs{\lam_j'} \lesssim \frac{1}{\lam_j} \Big( \bfd(t) + \zeta_{3, n} \Big) 
}
for some sequence $\zeta_{3, n} \to 0$ as $n \to \infty$. Then~\eqref{eq:lambda'-bound} follows by enlarging $\eps_n$. This completes the proof. 
\end{proof} 

\section{Conclusion of the proof} \label{sec:conclusion} 
  
For the remainder of the paper, when we write $[a_n, b_n] \in \calC_K(\eps_n, \eta)$ we we always assume that $\bfd(a_n) = \eps_n$, $\bfd(b_n) = \eta$ and $\bfd(t) \in [\eps_n, \eta]$ for all $t \in [a_n, b_n]$. This assumption is valid by Remark~\ref{rem:collision}. 
  
  
  \begin{lem} \label{lem:collision-duration}  
  If $\eta_0>0$ is small enough, then for any $\eta \in (0, \eta_0]$ there exist $\eps \in (0, \eta)$ and $C_u>0$ with the following property. If $[c, d] \subset [a_n, b_n]$, $\bfd(c)  \le \eps$ and $\bfd(d) \ge \eta$, then, 
  \EQ{
  (d- c)^{\frac{1}{2}} \ge C_u^{-1}  \lam_K(c)
  }
  \end{lem} 

\begin{proof} If not, there exists $\eta>0$, sequences $\eps_n \to 0$, $[c_n, d_n] \subset [a_n, b_n]$, and $C_n \to \infty$ so that $\bfd(c_n)  \le \eps_n$, $\bfd(d_n) \ge \eta$ and 
\EQ{ \label{eq:short-time} 
(d_n - c_n)^{\frac{1}{2}} \le C_n^{-1} \lam_K(c_n)
} 
We show that in this case $[c_n, d_n] \in \calC_{K-1}(\eps_n, \eta)$, which contradicts the minimality of $K$.

First, using~\eqref{eq:lambda'-bound} we see for all $j$, 
\EQ{ \label{eq:lambda'-dn} 
\abs{\lambda_j(t)^2 - \lambda_j(c_n)^2} \le C_0 (t- c_n)
}
for all $t \in [c_n, d_n]$. Hence, using the contradiction assumption~\eqref{eq:short-time}  we can ensure that for large enough $n$,  
\EQ{
\frac{3}{4} \le  \frac{ \lam_j(t)}{\lam_j(c_n)}   \le \frac{5}{4} 
}
for all $j = K, \dots, N$ and all $t \in [c_n, d_n]$.  Since $\bfd(c_n) \to 0$, it follows that, 
\EQ{ \label{eq:lamK0} 
\lim_{n \to \infty} \sup_{t \in [c_n, d_n]} \sum_{j=K}^{N} \Big( \frac{ \lam_j(t)}{\lam_{j+1}(t)} \Big)^k  = 0. 
} 
Next, since $ \bfd(c_n) \to 0$ we can find a sequence $r_n$ such that   
\EQ{ \label{eq:rn-choice} 
\lam_{K-1}(c_n) + (d_n - c_n)^{\frac{1}{2}} \ll r_n \ll \lam_K(c_n) \mand \lim_{n \to \infty} E( u(c_n) - u^*(c_n); \frac{1}{8} r_n, 8 r_n)  = 0.
}
Since $r_n \ll \al(t)$ we see that  $u(t, r) - u^*(t, r) = \chi_{\al(t)} u(t, r) + (1- \chi_{\al(t)}) m_{\De} \pi = u(t, r)$ for all  $r \in (1/8 r_n, 8 r_n)$. Letting $\phi(r)$ be a smooth bump equal to $1$ for  $r \in (1/4 , 4)$ and supported for $r \in(1/8, 8)$  with $|\phi'(r)| \le 16$, we apply~\eqref{eq:loc-en-ineq-1} with such a $\phi$  and deduce that for any $t \in [c_n, d_n]$, 
\EQ{
E( u(t); \frac{1}{4} r_n, 4 r_n) \le E( u(c_n); 1/8 r_n, 8 r_n) + C_0 \frac{ d_n - c_n}{r_n^2} 
}
and hence, 
\EQ{ \label{eq:rn-no-en} 
\lim_{n \to \infty}  \sup_{t \in [c_n, d_n]} E( u(t)- u^*(t); \frac{1}{4} r_n, 4 r_n) = 0.
}
 Next we claim that 
\EQ{ \label{eq:N-K+1-bubbles} 
\sup_{t \in [c_n, d_n]} E( u(t) - u^*(t); \frac{1}{4} r_n, \infty) \le (N-(K-1)) E(Q) + o_n(1) 
}
In the case $T_+< \infty$ we recall that $\al(t) = \al_n (T_+-t)^{\frac{1}{2}}$ and we write, 
\EQ{
 E( u(t) - u^*(t); \frac{1}{4} r_n, \infty)  =  E( u(t) - u^*(t); \frac{1}{4} r_n, \frac{1}{4}\al(t)) +   E( u(t) - u^*(t); \frac{1}{4} \al(t), \infty) 
}
Since $\al(t) \ge \rho(t)$ we have, 
\EQ{
\lim_{t \to \infty} E( u(t) - u^*(t); \frac{1}{4} \al(t), \infty)  = 0
}
Recalling that $u(t, r) - u^*(t, r) = u(t, r)$ for all $r \le \al(t)$ we again apply~\eqref{eq:loc-en-ineq-1} with the cut-off function $\phi(t, r) = (1- \chi_{4 r_n}(r)) \chi_{\frac{1}{4}\al(t)}(r)$. Since $\frac{\ud}{\ud t} \phi(t, r) \le 0$ we use ~\eqref{eq:loc-en-ineq-1} to deduce that for all $t \in [c_n, d_n]$, 
\EQ{
 E( u(t) - u^*(t); \frac{1}{4} r_n, \frac{1}{4}\al(t)) \le  E( u(c_n) - u^*(c_n); \frac{1}{8} r_n, \frac{1}{2}\al(t)) + C_0\frac{ d_n - c_n}{r_n^2} 
}
and the right hand side tends to zero as $n \to \infty$, proving~\eqref{eq:N-K+1-bubbles} in the case $T_+< \infty$.  If $T_+ = \infty$, we use the same argument, but without the need to truncate at $\al(t)$ since we have $u^*(t) := 0$. 

Next, using~\eqref{eq:lambda'-dn} with $j = K-1$ gives,  
\EQ{
\sup_{t \in [c_n, d_n]} \abs{ \lam_{K-1}(t)^2 - \lam_{K-1}(c_n)^2} \lesssim d_n - c_n, 
}
and hence 
\EQ{
\sup_{t \in [c_n, d_n]} \frac{\lam_{K-1}(t)}{ r_n} \lesssim \frac{\lam_{K-1}(c_n)}{ r_n} + \frac{(d_n - c_n)^{\frac{1}{2}}}{r_n} \to 0 \mas n \to \infty
}
given our choice of $r_n$ in~\eqref{eq:rn-choice}. Using all of the above, we can find $m_n \in\Z$ so that defining, 
\EQ{
v(t):= (1- \chi_{r_n}) (u(t) - u^*(t)) + \chi_{r_n} m_n \pi
}
we have $v(t) \in \E_{m_n, m_\De}$ for $t \in [c_n, d_n]$ and such that 
\EQ{
\| v(t) - \calQ( m_\De, \iota_{K}, \dots, \iota_{N}, \lam_K(t), \dots \lam_{N}(t)) \|_{\E} + \sum_{j = K}^{N-1} \Big(\frac{ \lam_{j}(t)}{ \lam_{j+1}(t)} \Big)^{\frac{k}{2}} \lesssim \eta
}
It follows that $\bfd_{m_\De, N-K+1}(v(t)) \lesssim \eta$ and we can apply Lemma~\ref{lem:mod-static} to find modulation parameters $\vec {\ti\iota} \in \{-1, 1\}^{N-K+1}$,  $\ti \lam_{K}(t), \dots, \ti \lam_{N}(t)$ and $h(t)$ defined by 
\EQ{
h(t) =  v(t) - \calQ( m_\De, \iota_{K}, \dots, \iota_{N}, \ti \lam_K(t), \dots, \ti  \lam_{N}(t))
}
so that 
\EQ{
0= \La \calZ_{\U{\ti \lam_j(t)}} \mid h(t) \Ra , \mand \| h(t) \|_{\E} + \sum_{j=K}^{N-1}  \Big(\frac{\ti  \lam_{j}(t)}{\ti  \lam_{j+1}(t)} \Big)^{\frac{k}{2}} \lesssim \eta
}
In fact, using~\eqref{eq:lamK0} and the fact that the $\ti \lam_j(t)$ satisfy $| \ti \lam_j(t)/ \lam_j(t) - 1| \lesssim \eta$, we have,  
\EQ{
\lim_{n \to \infty} \sup_{t \in [c_n, d_n]} \sum_{j=K}^{N-1}  \Big(\frac{\ti  \lam_{j}(t)}{\ti  \lam_{j+1}(t)} \Big)^{\frac{k}{2}} = 0. 
}
And, thus, using~\eqref{eq:g-bound-A} along with~\eqref{eq:N-K+1-bubbles} we have the bound,  
\EQ{
\| h(t) \|_{\E}  \lesssim \sum_{j=K}^{N-1}  \Big(\frac{\ti  \lam_{j}(t)}{\ti  \lam_{j+1}(t)} \Big)^{\frac{k}{2}} + o_n(1) 
}
and thus $\lim_{n \to \infty} \sup_{t \in [c_n, d_n]} \| h(t) \|_{\E}  = 0$ as well. Letting $\rho_{K-1}(t) := r_n$ for $t \in [c_n, d_n]$ we have proved that 
\EQ{
\lim_{n \to \infty} \sup_{t \in [c_n, d_n]} \bfd(t; \rho_{K-1}(t))  = 0
}
which means we can find $\ti \eta>0$, $\ti \eps_n \to 0$ such that $[c_n, d_n] \in \calC_{K-1}( \ti \eps_n, \ti \eta)$ contradicting the minimality of $K$. 
\end{proof} 

\begin{lem}\label{lem:cndn} Let $\eta_0>0$ be as in Lemma~\ref{lem:collision-duration},  $\eta \in (0, \eta_0]$, $\eps_n \to 0$ be some sequence, and let $[a_n, b_n] \in \calC_K(\eps_n, \eta)$. Then, there exist $\eps \in (0, \eta)$,  $n_0 \in \N$,  and $c_n,  d_n \in (a_n, b_n)$ such that for all $n \ge n_0$, we have 
\EQ{ \label{eq:d>eps} 
\bfd(t) \ge \eps, \quad \forall \, \, t \in [c_n, d_n], 
}
\EQ{\label{eq:dn-cn} 
d_n - c_n = \frac{1}{n} \lam_K(c_n)^2, 
}
and 
\EQ{ \label{eq:lamKcn} 
\frac{1}{2} \lam_K(c_n) \le \lam_K(t) \le 2\lam_K(c_n) \quad \forall\, \, t \in [c_n, d_n]. 
} 
\end{lem} 

\begin{proof} 
Choose $\eps>0$ so that Lemma~\ref{lem:collision-duration} holds and define $c_n:= \sup\{t \in [a_n, b_n] \mid \bfd(t) \le \eps\}$. Then $\bfd(c_n) = \eps$ and by Lemma~\ref{lem:collision-duration} we have 
\EQ{
b_n - c_n \ge C_u^{-1} \lam_K(c_n). 
}
We then let $d_n:= c_n + \frac{1}{n} \lam_K(c_n)^2$ and for $n$ sufficiently large we have $d_n < b_n$. Then by~\eqref{eq:lambda'-bound} we have, 
\EQ{
\Big| \frac{ \lam_K(t)^2}{\lam_K(c_n)^2} - 1 \Big| \lesssim \frac{ d_n- c_n}{\lam_K(c_n)}  = \frac{1}{n}. 
}
from which~\eqref{eq:lamKcn} follows. 
\end{proof} 

\begin{lem} \label{lem:delta-to-d} There exists $\eta_1>0$ with the following property. Let $\eta \in (0, \eta_1]$, $\eps_n \to 0$ and let $[a_n, b_n] \in \calC_K(\eps_n, \eta)$. If $\{s_n \}_{n}$ and $\{r_n\}_n$ are any sequences such that $s_n \in [a_n, b_n]$ for all $n$, $1 \ll r_n \ll \lam_{K+1}(s_n)/\lam_{K}(s_n)$, and $ \lim_{n \to \infty} \bs \delta_{r_n \lam_K(s_n)}(u(s_n)) = 0$, then $\lim_{n \to \infty} \bfd(s_n) = 0$. 
\end{lem} 

\begin{proof}
Let $R_n$ be a sequence such that $r_n\lam_K(s_n) \ll R_n \ll \lambda_{K+1}(s_n)$.
Without loss of generality, we can assume $\nu(s_n)  \le  R_n \le \al(s_n)$, since it suffices to replace $R_n$ by $\nu(s_n)$
for all $n$ such that $R_n < \nu(s_n)$. If $K = N$ we can similarly ensure that $R_n \le \al(s_n)$. 
Let $M_n, m_n, \vec\sigma_n \in \{-1, 1\}^{M_n}, \vec \mu_n \in (0, \infty)^{M_n}$ be parameters such that
\begin{equation}
\label{eq:conv-delta-iii}
\| u(t_n) - \calQ( m_n, \vec \sigma_n, \vec \mu_n) \|_{H( r \le r_n\lambda_K(s_n))}^2 + \sum_{j = 1}^{M_n} \Big(\frac{ \mu_{n, j}}{ \mu_{n, j+1}}\Big)^k  + \frac{\mu_{n, M_n}}{r_n \lam_{K}(s_n)} \to 0,
\end{equation}
which exist by the definition of the localized distance function \eqref{eq:delta-def}.  Since $\bfd(t) \le \eta$ on $[a_n, b_n]$ we can choose $\eta_1>0$ sufficiently small so that, 
\begin{equation}
\Big(K - \frac 12\Big)E( Q) \leq \liminf_{n\to\infty}E( u(s_n); 0, r_n\lam_K(s_n)) \leq \limsup_{n\to\infty}E(u(s_n); 0, r_n\lam_K(s_n)) \leq \Big(K+\frac{1}{2}\Big) E( Q), 
\end{equation} 
after noting that the radiation $u^*$ is negligible on the region $r \le r_n\lam_K(s_n)$. 
Hence,  $M_n = K$ for $n$ large enough.
We set $\mu_{n, j} := \lam_j(s_n)$ and $\s_{n, j} := \iota_j$ for $j > K$.
We claim that
\begin{equation}
\label{eq:d-conv-0}
\lim_{n\to \infty}\bigg(\|  u(s_n) -  u^* - \calQ(m_\Delta, \vec\s_n, \vec\mu_n) \|_{\cE}^2 + \sum_{j=1}^{N}\Big(\frac{ \mu_{n, j}}{\mu_{n, j+1}}\Big)^{k}\bigg) = 0.
\end{equation}
By the definition of $\bfd$, the proof will be finished.
First, recall that $\mu_{n, K} \ll r_n\mu(t_n)$, so $\mu_{n, K} / \mu_{n, K+1} \to 0$.
In the region $r \leq r_n\lam_K(s_n)$, convergence follows from \eqref{eq:conv-delta-iii},
since the energy of the exterior bubbles asymptotically vanishes there.
In the region $r \ge R_n$, the energy of the interior bubbles vanishes, hence it suffices to apply~\eqref{eq:N-K-converge}.
In particular, by the above and~\eqref{eq:N-K-bubbles}, 
\begin{equation}
\lim_{n\to\infty}E( u(s_n); 0, r_n\lam_K(s_n)) = KE(Q), \qquad \lim_{n\to\infty} E( u(s_n); R_n, \infty) = (N-K)E(Q) + E( u^*),
\end{equation}
which implies
\begin{equation}
\lim_{n\to\infty} E( u(s_n); r_n\lambda_K(s_n), R_n) = 0,
\end{equation}
and \eqref{eq:H-E-comp} yields convergence of the error also in the region $r_n\lambda_K(s_n) \leq r \leq R_n$.
\end{proof}

\begin{proof}[Proof of Theorem~\ref{thm:main}]
Assume the theorem is false and let $[a_n, b_n] \in \calC_K(\eps_n, \eta)$ be a sequence of disjoint collision intervals given by Lemma~\ref{lem:modulation}, and $\eta>0$ is sufficiently small so that Lemma~\ref{lem:collision-duration} and Lemma~\ref{lem:delta-to-d} hold. Let $\eps>0$, $n_0$, and $[c_n, d_n]$ be as in Lemma~\ref{lem:cndn}. 

We claim that there exists $c_0>0$ such that for every $n \ge n_0$, 
\EQ{ \label{eq:tension-lower} 
\inf_{t \in [c_n, d_n]} \lam_K(t)^2 \| \p_t u(t) \|_{L^2}^2 \ge c_0. 
}
If not, we could, after passing to a subsequence, find a sequence $s_n \in [c_n, d_n]$ such that
\EQ{
\lim_{n \to \infty} \lam_K(s_n) \| \p_t u(s_n) \|_{L^2}  = 0
}
But then an application of Lemma~\ref{lem:compact} gives a sequence $r_n \to \infty$ such that, after passing to a further subsequence,  $ \lim_{n \to \infty} \bs \delta_{r_n \lam_K(s_n)} ( u(s_n)) = 0$.  But then Lemma~\ref{lem:delta-to-d} gives that $\lim_{n \to \infty} \bfd(s_n) = 0$, which contradicts~\eqref{eq:d>eps}. Thus~\eqref{eq:tension-lower} holds. 

Therefore, using~\eqref{eq:tension-lower},~\eqref{eq:lamKcn}, and~\eqref{eq:dn-cn} we have 
\EQ{
\sum_{n \ge n_0} \int_{c_n}^{d_n} \| \p_t u(t) \|_{L^2}^2 \, \ud t  \ge \frac{c_0}{4} \sum_{n \ge n_0} \int_{c_n}^{d_n} \lam_K(c_n)^{-2} \,  \ud t \ge \frac{c_0}{4} \sum_{n \ge n_0} n^{-1} = \infty.
}
On the other hand,  by~\eqref{eq:tension-L2} and the fact that the $[c_n, d_n]$ are disjoint, we have, 
\EQ{
\sum_{n \ge n_0} \int_{c_n}^{d_n} \| \p_t u(t) \|_{L^2}^2 \, \ud t  \le \int_0^{T_*} \| \p_t u(t) \|_{L^2}^2 \, \ud t < \infty, 
}
which is a contradiction. 
\end{proof} 

\bibliographystyle{plain}
\bibliography{HMHF}

\bigskip
\centerline{\scshape Jacek Jendrej}
\smallskip
{\footnotesize
 \centerline{CNRS and LAGA, Universit\'e  Sorbonne Paris Nord}
\centerline{99 av Jean-Baptiste Cl\'ement, 93430 Villetaneuse, France}
\centerline{\email{jendrej@math.univ-paris13.fr}}
} 
\medskip 
\centerline{\scshape Andrew Lawrie}
\smallskip
{\footnotesize
 \centerline{Department of Mathematics, Massachusetts Institute of Technology}
\centerline{77 Massachusetts Ave, 2-267, Cambridge, MA 02139, U.S.A.}
\centerline{\email{alawrie@mit.edu}}
} 


\end{document}